\documentclass[11pt]{amsart}
\usepackage{geometry}                
\usepackage{graphicx}
\usepackage{amssymb}
\usepackage{epstopdf}
\usepackage{amsmath}
\usepackage{color}
\usepackage{hyperref}
\usepackage{pdfsync}
\usepackage{tikz}
 
\usepackage[all]{xy}
\xyoption{matrix}
\xyoption{arrow}

\newtheorem{thm}{Theorem}[section]
\newtheorem{lem}[thm]{Lemma}
\newtheorem{cor}[thm]{Corollary}
\newtheorem{prop}[thm]{Proposition}

\newtheorem{innercustomthm}{Theorem}
\newenvironment{customthm}[1]
{\renewcommand\theinnercustomthm{#1}\innercustomthm}
  {\endinnercustomthm}
  
  \newtheorem{innercustomlem}{Lemma}
\newenvironment{customlem}[1]
	{\renewcommand
	\theinnercustomlem{#1}\innercustomlem}
	{\endinnercustomlem}
 
\theoremstyle{definition}
\newtheorem{defn}[thm]{Definition}
\newtheorem{eg}[thm]{Example}

\theoremstyle{remark}
\newtheorem{rem}[thm]{Remark}
 
\numberwithin{equation}{section}
 
 \tikzset{help lines/.style={step=#1cm,very thin, color=gray},
help lines/.default=.5} 
\tikzset{thick grid/.style={step=#1cm,thick, color=gray},
thick grid/.default=1} 

\DeclareMathOperator{\coker}{coker}
\DeclareMathOperator{\Hom}{Hom}%
\DeclareMathOperator{\Ext}{Ext}%
\DeclareMathOperator{\End}{End}%
\DeclareMathOperator{\colim}{colim}%

\DeclareMathOperator{\undim}{\underline{dim}}

\newcommand{\cA}{\ensuremath{{\mathcal{A}}}}

\newcommand{\cB}{\ensuremath{{\mathcal{B}}}}
\newcommand{\cC}{\ensuremath{{\mathcal{C}}}}
\newcommand{\cD}{\ensuremath{{\mathcal{D}}}}
\newcommand{\cE}{\ensuremath{{\mathcal{E}}}}
\newcommand{\cF}{\ensuremath{{\mathcal{F}}}}

\newcommand{\cL}{\ensuremath{{\mathcal{L}}}}
\newcommand{\cM}{\ensuremath{{\mathcal{M}}}}
\newcommand{\nM}{\ensuremath{{\mathcal{M}}}}

\newcommand{\cT}{\ensuremath{{\mathcal{T}}}}

\newcommand{\cW}{\ensuremath{{\mathcal{W}}}}

\newcommand{\mat}[1]{{\left[
\begin{matrix}#1\end{matrix}
\right]}}

\title{Linearity of stability conditions}

\author{Kiyoshi Igusa}
\address{Department of Mathematics, Brandeis University, Waltham, MA 02454}\email{igusa@brandeis.edu}



\begin{document}

\begin{abstract} For modules over an artin algebra a linear stability condition is given by a ``central charge'' and a nonlinear stability condition is given by the wall-crossing sequence of a ``green path''. Finite Harder-Narasimhan stratifications of the module category, maximal forward hom-orthogonal sequences and maximal green sequences, defined using Fomin-Zelevinsky quiver mutation are shown to be equivalent to finite nonlinear stability conditions when the algebra is hereditary.

This is the first of a series of three papers whose purpose is to determine all maximal green sequences of maximal length for quivers of affine type A and determine which are linear. See [1]. 
\end{abstract}

\maketitle




Stability conditions and Harder-Narasimhan filtrations are very active areas of research. Some interesting examples are: \cite{T}, \cite{GSZ}, \cite{FR}, \cite{KQ}, \cite{Z}, \cite{Q} and, by the author, \cite{BHIT}, \cite{IOTW}, \cite{IOTW2}, \cite{PartII}. Key references are \cite{King}, \cite{Rudakov}, \cite{HN}, \cite{B}, \cite{B2}, \cite{R}, \cite{Keller}, \cite{KS}, \cite{DW}. Stability functions on quiver representations were introduced by King \cite{King}. This was generalized to abelian categories by Rudakov \cite{Rudakov} who also proved the Harder-Narasimhan property \cite{HN} under a finiteness property which includes categories of vector bundles which he had been studying earlier and representations of finite dimensional algebras which we consider in this paper. Bridgeland \cite{B} extended linear stability conditions and HN-filtrations to triangulated categories. Reineke \cite{R} showed that the quantum Donaldson-Thomas invariant can be computed using a linear stability function. He wrote it as a sum of terms one for each stable modules of the stability condition. So, it became important to know which sets of modules are given by linear stability conditions. In \cite{Keller} {Keller} uses nonlinear stability conditions (equivalent to maximal green sequences for quivers with potential) to give a formula for the refined Donaldson-Thomas invariants of Kontsevich-Soibelman \cite{KS}. Bridgeland considers nonlinear stability conditions in more general contexts in \cite{B2}.

Derksen and Weyman showed that stability conditions, in terms of semi-invariants, can be used to obtain canonical representations of quivers \cite{DW2} and they also used it to give a new proof of the saturation conjecture for Littlewood-Richardson coefficients \cite{DW}. The Stability Theorem in \cite{DW} gives the precise relation between semi-invariants and stability conditions. This was later extended to the ``virtual Stability Theorem'' in \cite{IOTW} and, in the modulated case, in \cite{IOTW2}. In \cite{BHIT} these semi-invariant stability conditions were used to prove two conjectures about maximal green sequences. And they will be used in (b) below.

Let $\Lambda$ be a finite dimensional algebra over a field $K$. We will always use $n$ to denote the number of simple $\Lambda$-modules. Then every $\Lambda$-module $M$ has \emph{dimension vector}
\[
	\undim M:=(x_1,\cdots,x_n)\in\mathbb N^n\subset \mathbb Z^n
\]
where $x_i$ is the number of times the $i$th simple $\Lambda$-module $S_i$ occurs in the composition series of $M$. We consider five different concepts of stability for $\Lambda$-modules and whether these stability conditions are linear and/or green and/or finite.

\begin{enumerate}
\item[(a)] Stability functions\footnote{Stability functions in the same sense as in \cite{BST2}.} $Z_t:K_0\Lambda=\mathbb Z^n\to \mathbb C$.
\item[(b)] Wall crossing sequences $D(M_i)$ of a ``generic path'' $\gamma:\mathbb R\to \mathbb R^n$.
\item[(c)] Harder-Narasimhan stratifications of $mod\text-\Lambda$ (Definition \ref{def: HN-stratification}).
\item[(d)] Maximal forward hom-orthogonal sequences of Schurian modules. (Theorem \ref{main theorem}(d)) A module $M$ is called \emph{Schurian} if its endomorphism ring is a division algebra. 
\item[(e)] Sequences of $c$-vectors for ``reddening sequences'' given by Fomin-Zelevinsky mutation (Theorems \ref{thm: exchange matrix transforms}, \ref{Rem: nonlinear stability functions and reddening sequences}.
\end{enumerate}
\noindent These five stability concepts fall into three sets. 
\begin{enumerate}
\item (a) and (b) are easily seen to be equivalent in both linear and nonlinear cases. Furthermore, $Z_\bullet$ is ``green'' if and only if the corresponding wall crossing sequence is green where \emph{green} means the directional velocity of the path $\gamma_Z$ in the direction $\undim M_i$ is positive whenever $\gamma_Z(t_i)\in D(M_i)$ (Definition \ref{def of red or green path}, Lemma \ref{lem2}).
\item (c) and (d) are show to be equivalent in the finite case (when there are only finitely many strata in the HN-stratification and only one Schurian module in each stratum) (Theorem \ref{thm: HN statification is Schurian hom orthog}). Also, (a) implies (c) in the green case (Theorem \ref{thm: HN filtration}).
\item Fomin-Zelevinsky mutation (e) only makes sense when $\Lambda$ is a cluster-tilted algebra. This includes hereditary algebras. For these algebras, we claim that (e) is equivalent to the finite case of (b) and that, in the green case, (e) is equivalent to (c) and (d). However, we only prove these statements in the hereditary case in this paper and for cluster-tilted algebras of finite type in \cite{PartII}.
\end{enumerate}
To summarize:
\[
\xymatrixrowsep{10pt}\xymatrixcolsep{35pt}
\xymatrix{
(a) \ar@{<=>}[r] & (b)\ar@{=>}[r]^{green}& (c) \ar@{<=>}[r]^{finite} &(d)\ar@{=>}[r]^{hereditary} &(e)\ar@{<=>}[r]^{hereditary}_{finite}&(b)\\
	}
\]
In the finite, green, hereditary case, all five conditions are equivalent. So, we have five equivalent ways to describe the same sequence of Schurian $\Lambda$-modules. Furthermore, these modules will be uniquely determined by their dimension vectors (up to isomorphism).

\begin{customthm}{C}\label{main theorem}
Let $\Lambda$ be a finite dimensional hereditary algebra over a field $K$. Let $\beta_1,\cdots,\beta_m\in \mathbb N^n$ be any finite sequence of nonzero, nonnegative integer vectors. Then the following are equivalent.
\begin{enumerate}
\item[(a)] There is a nonlinear stability function $Z_t:K_0\Lambda\to \mathbb C$ which is green and has exactly $m$ semistable pairs $(M_i,t_i)$ with $t_1<t_2<\cdots<t_m$ so that $\undim M_i=\beta_i$ for all $i$.
\item[(b)] There is a generic green path $\gamma:\mathbb R\to\mathbb R^n$ which crosses the walls $D(M_i)$, $i=1,\cdots,m$ in that order, and no other walls, so that $\undim M_i=\beta_i$ for all $i$.
\item[(c)] There exist $\Lambda$-modules $M_1,\cdots,M_m$ with $\undim M_i=\beta_i$ which form a finite Harder-Narasimhan system for $\Lambda$. 
\item[(d)] There exist Schurian $\Lambda$-modules $M_1,\cdots,M_m$ with $\undim M_i=\beta_i$ so that \begin{enumerate}
\item[(1)] $\Hom_\Lambda(M_i,M_j)=0$ for $i<j$.
\item[(2)] No other modules can be inserted into the sequence preserving (1).
\end{enumerate}
\item[(e)] There is a maximal green sequence for $\Lambda$ of length $m$ whose $i$th mutation is at the $c$-vector $\beta_i$. 
\end{enumerate}
\end{customthm}

The equivalent $(a)\Leftrightarrow (b)$ is Proposition \ref{prop: (a) iff (b)}, a special case of Theorem \ref{thm: (a) iff (b)}. Theorem \ref{thm: HN statification is Schurian hom orthog} proves $(c)\Leftrightarrow (d)$. Corollary \ref{thm: nonlinear Z correspond to MGSs} shows that $(a)\Leftrightarrow (e)$. Theorem \ref{thm: hom-orthog = mgs} shows $(d)\Leftrightarrow (e)$. Much of this is well-known, e.g., $(c)\Leftrightarrow (e)$ is proved in the Dynkin case in \cite{Q}. The main new idea in this paper is (d) which is a very useful formulation of stability which will be used in the next paper \cite{PartII} to obtain new results about maximal green sequences of maximal length for cluster-tilted algebras. (In \cite{PartII} it is shown that (b), (d) and (e) are equivalent for cluster-tilted algebras of finite type over an algebraically closed field.) Note that Theorem \ref{thm: HN statification is Schurian hom orthog} proves that $(c)\Leftrightarrow (d)$ holds over any finite dimensional algebra. Keller's survey article \cite{KD} also gives a very useful summary of known results.

This paper is the first of three papers motivated by the question of which stability conditions are ``linear''. By this we mean that it is given by a linear stability function $Z:K_0\Lambda\to \mathbb C$ also called a ``central charge''. This question originates in a conjecture by Reineke \cite{R} who asks: For any Dynkin quiver $Q$, there is a (classical) slope function $\mu$ whose corresponding central charge $Z$ makes all $KQ$-modules stable? See Remark \ref{rem: slope function}. Qiu partially solved this problem in \cite{Q} where he shows that, for some orientation for each Dynkin, there is a central charge making all indecomposable modules stable.

{This paper, together with \cite{AI}, \cite{PartII}} addresses the extension of Reineke's question to a hereditary algebra $\Lambda$ of affine type $\tilde A_{n-1}$ and to cluster-tilted algebras of finite type. We know by \cite{BDP}, \cite{BHIT} that there are only finitely many maximal green sequences. So, there is a longest one. Using all the equivalent formulations, the problem now comes in three parts.
\begin{enumerate}
\item Find the maximum length of all maximal green sequences for $\Lambda$. Equivalently, find the maximum size of a finite HN-system for $\Lambda$. Call this $L$.
\item Describe all $L$ element sets of $\Lambda$-modules which can be arranged into a maximal hom-orthogonal sequence.
\item Which of these sets is the set of stable modules of a linear stability function $Z:K_0\Lambda\to\mathbb C$? 
\end{enumerate}

In \cite{AI} we completely solve this problem in the case $\tilde A_{a,b}$, the tame hereditary algebra given by a cyclic quiver with $a$ arrows pointed one way and $b=n-a$ arrows pointed the other way. For example, $L=\binom n2+ ab$. When $(a,b)=(n-1,1)$ this is already known \cite{Kase}. The case $(a,b)=(n,0)$ is also solved in \cite{AI}. This is cluster-tilted of type $D_n$. In this case $L=\binom n2+n-1$, there are $n$ sets and all are given by linear stability functions.

The purpose of the present paper and the next \cite{PartII} is to lay the foundations for these result. This paper addresses different notions of stability for hereditary algebras and \cite{PartII} addresses different notions of stability for cluster-tilted algebras of finite type.

\vskip.2cm 

Contents of the paper:\vskip.2cm

Section \ref{sec1}. We discuss the definition of a linear stability condition for representations of a finite dimensional algebra, or, more generally, for nilpotent representations of any modulated quiver. It is immediate (Theorem \ref{thm: path for Z goes through D(M)}) that a linear stability function $Z$ corresponds to a linear path through semistability sets $D(M)$, also called \emph{walls}, for $Z$-semistable modules $M$. We also go through one example, the cyclic quiver with three vertices module $rad^k$ for various $k\ge2$.

{Section \ref{sec2}. We define nonlinear stability functions $Z_\bullet$. The corresponding nonlinear paths cross the walls $D(M)$ for $Z_\bullet$-semistable modules $M$. Conversely, we show, in Theorem \ref{thm: (a) iff (b)}, that any ``reddening path'' (Definition \ref{def: reddening path}) comes form a nonlinear stability function. Section \ref{ss: wide subcategory proof} gives a short proof of the well-known fact that, for any subset $S\subset\mathbb R^n$, $\cW(S)$, the full subcategory of $mod\text-\Lambda$ of all $M$ so that $S\subseteq D(M)$, is a wide subcategory of $mod\text-\Lambda$. 
Section \ref{ss: HN filtration} proves that any green path (Definition \ref{def: reddening path}) gives a Harder-Narasimhan stratification of $mod\text-\Lambda$ (Definition \ref{def: HN-stratification}) for any finite dimensional algebra $\Lambda$.}


Section \ref{sec3}. We define a ``finite HN-system'' (Definition \ref{def: finite HN-system}) for $mod$-$\Lambda$ to be a special case of a finite HN-stratification and we show that it is equivalent to a ``maximal forward hom-orthogonal sequence'' (Definition \ref{def: forward hom-orthogonal}) of Schurian $\Lambda$-modules $M_1,\cdots,M_m$. We observe that a finite, green, nonlinear stability function $Z_\bullet$ gives such a finite HN-system. This is the implication $finite\,green\,(a)\Rightarrow finite\,(c)\Leftrightarrow (d)$ mentioned above. 

Section \ref{sec4}. We show that, for any hereditary algebra $\Lambda$, the finite nonlinear stability functions give ``reddening sequences'' and all reddening sequences are given in this way. As a special case, finite green stability functions give ``maximal green sequences'' and all maximal green sequences are given in that way. We use the semi-invariant wall-crossing description of maximal green sequences from \cite{IOTW2} to make this correspondence. We also show that, in the hereditary case, maximal green sequences are equivalent to maximal forward hom-orthogonal sequences of Schurian modules. This shows that all five notions of stability in Theorem \ref{main theorem} above are equivalent in the finite, green hereditary case. 

Section \ref{sec5} proves the crucial Theorem \ref{lemma A} which implies that, for any hereditary $\Lambda$, those semistability sets $D(M)$ which are not equal to the semi-invariant sets from \cite{IOTW2}, i.e., the nonexceptional semistability sets, are never encountered in a finite path. This implies $(b)\Leftrightarrow(e)$ since that statement was proved in \cite{IOTW2} using semi-invariant walls.

Finally, we should mention that many of these results have been extended to arbitrary finite dimensional algebras using $\tau$-tilting by Br\"ustle, Smith and Treffinger \cite{BST}. Also, recently, Li and Liu \cite{LL} have extended these results to any abelian length category and have obtained a numerical criterion for which maximal green sequences (defined as maximal chains of torsion classes) are linear.

%
%

\section{Linear stability functions}\label{sec1}

Linear stability conditions are given in two equivalent ways: by a ``central charge'' $Z:K_0\Lambda\to\mathbb C$ or by a linear ``green path'' $\gamma:\mathbb R\to\mathbb R^n$. A $\Lambda$-module $M$ is $Z$-semistable if and only if the corresponding path $\gamma_Z$ crosses the semistability set $D(M)$. Following a suggestion by Yang-Hui He, we treat $\Lambda$ as a modulated quiver with unspecified relations. This is equivalent to considering nilpotent representations of the quiver.



\subsection{Nilpotent representations}\label{ss: nilpotent representations}

Let $K$ be any field and let $\nM=(\{F_i\},\{M_{ij}\})$ be a modulated quiver over $K$, possibly with oriented cycles. This is define by a finite set $\{F_1,\cdots,F_n\}$ of finite dimensional division algebras over $K$ together with finite dimensional $F_i$-$F_j$ bimodules $M_{ij}$. A finite dimensional \emph{representation} $X$ of $\nM$ is defined to be a collection of finite dimensional right $F_i$-vector spaces $X_i$ together with $F_j$ linear maps
\[
	X_i\otimes_{F_i}M_{ij}\to X_j
\]
for all $i,j$. Representations of $\nM$ are the same as right modules over the \emph{tensor algebra} $T\nM$ of $\nM$ which is defined to be the direct sum $T\nM=\coprod_{k\ge 0}T_k\nM$ where $T_0\nM=\prod F_i$ and, for $k\ge1$, $T_k\nM$ is the direct sum of all tensor paths of length $k$:
\[
	T_k\nM=\coprod M_{j_0j_1}\otimes_{F_{j_1}}M_{j_1j_2}\otimes_{F_{j_2}}\cdots \otimes_{F_{j_{k-1}}}M_{j_{k-1}j_k}
\]
Let $R\nM=\coprod_{k\ge 1}T_k\nM$ and let $nmod\text-\nM$ be the category of those finitely generated right $T\nM$ modules on which $R\nM$ act nilpotent. We call such modules \emph{nilpotent} $\nM$-modules and we also refer to the corresponding representations of $\nM$ as \emph{nilpotent}.

Each nilpotent module $X$ is a module over $T\nM/R\nM^m$ for some $m$ and thus a module over a finite dimensional algebra over $K$. For example, if $\nM$ is given by a loop $a$ at a single vertex then $T\nM=K[a]$ and $nmod\text-\nM$ is the category of all finite dimensional vector spaces $X$ together with a nilpotent endomorphism $a$. Then $nmod\text- K[a]$ has only one simple module corresponding to the maximal ideal $R=(a)$. In general, the category $nmod\text-\nM$ has only $n$ simple modules $S_1,\cdots,S_n$ which are one-dimensional over $F_1,\cdots,F_n$ respectively.

A nilpotent module $X$ gives a nilpotent representation of the modulated quiver $\nM$ in a standard way by letting $X_i$ be the $F_i$-vector space $X_i=Xe_i$ where $e_i$ is unity in $F_i$ and, for every $i,j$ taking $X_i\otimes M_{ij}\to X_j$ to be the $F_j$-linear map given by the action of $M_{ij}\subseteq T\nM$ on $X$.

The category $nmod\text-\nM$ is based on a suggestion by Yang-Hui He at a conference at the Chinese University of Hong Kong. This category is the union or colimit:
\[
	nmod\text-\nM=\bigcup mod\text-T\nM/J=\colim mod\text-T\nM/J
\]
of the module categories $mod\text-T\nM/J$ of all finite dimensional algebras of the form $\Lambda=T\nM/J$ where $J$ is an {admissible} ideal $J\subseteq R\nM^2$ where \emph{admissible} means that $J$ is a two sided ideal in $T\nM$ so that $R\nM^k\subseteq J\subseteq R\nM^2$ for some $k\ge 2$. Every nilpotent representation of $\nM$ is a $T\nM/J$-module for some $J$. Conversely, every module over $\Lambda=T\nM/J$ is an object of $nmod\text-\nM$ and every subquotient module of such a nilpotent module is also a $\Lambda$-module.

Some readers were confused by the relationship between $nmod\text-\nM$ and $mod\text-\Lambda$. So, we separated them into two cases and point out differences along the way. We use the notation $\cL$ to denote either $nmod\text-\nM$ or $mod\text-\Lambda$ for some finite dimensional algebra $\Lambda$. We always assume that $\cL$ has $n$ simple modules. All objects of $\cL$ are finite dimensional.


\subsection{Linear stability function $Z$}

{For $\cL=nmod$-$\nM$ or $mod\text-\Lambda$, let $K_0\cL$ be the Grothendieck group of the category $\cL$. Then $K_0\cL\cong \mathbb Z^n$ where we identify $[M]\in K_0\cL$ with $\undim M\in\mathbb Z^n$ giving the composition series of $M$. Thus}
\begin{equation}\label{dot prod equation}
	\dim_KM=(\dim_K S_1,\cdots,\dim_K S_n)\cdot \undim M.
\end{equation}
{We write $K_0\Lambda$ and $K_0\nM$ for $K_0\cL$ in the two cases.} Note that $K_0\nM=K_0 \Lambda$ for any $\Lambda=T\nM/J$.

 \begin{defn} A \emph{linear stability function} for $\cL$ is an additive map:
\[
	Z:K_0\cL\to \mathbb C
\]
which we write as:
\[
	Z(x)=a\cdot x+ib\cdot x = r(x)e^{i\theta(x)}
\]
where $a\in\mathbb R^n$, $b\in (0,\infty)^n$ are fixed and $0<\theta(x)<\pi$. For any $M\in \cL$, let
\[
	\mu(M)=\mu_Z(M):=\frac{a\cdot\undim M}{b\cdot\undim M}=\cot \theta( M)
\]
where $\theta( M)=\theta(\undim M)$. This is called the \emph{slope} of $M$. Note that $\mu(M)$ is a monotonically decreasing function of $\theta( M)$. A nonzero object $M\in\cL$ is called $Z$-\emph{semistable}, resp. $Z$-\emph{stable}, if $\mu(M')\ge \mu(M)$, resp. $\mu(M')> \mu(M)$, for all nonzero submodules $M'\subsetneq M$.
\end{defn}

We should point out that this definition of stability is the opposite of that in some other papers, such as \cite{BST2} and \cite{IOTW} where modules are stable when their slopes are larger than the slopes of proper submodules.

\begin{rem}\label{rem: slope function}
An important special case is the \emph{classical choice} $b=(\dim_K S_1,\cdots,\dim_K S_n)$. For any $a\in \mathbb R^n$ the resulting function $\mu(M)=a\cdot \undim M/\dim_K M$ is a \emph{classical slope function}. Reineke's original conjecture \cite{R} is that, for every Dynkin quiver, there is a classical slope function making all indecomposable modules stable.
\end{rem}

We observe that simple modules are always stable. Often the restriction on $\theta$ is taken to be $0\le \theta<\pi$ and $Z:K_0\Lambda\to\mathbb C$ is called a \emph{central charge}. We take $\theta>0$ so that the slope function $\mu$ is defined.


\subsection{Linear green path $\gamma$}

For any linear stability function $Z:K_0\cL\to \mathbb C$ we will show that there is a corresponding linear path $\gamma_Z:\mathbb R\to\mathbb R^n$ which crosses the wall $D(M)$, defined below, whenever $M$ is $Z$-semistable.

\begin{defn}
For any object $M\in \cL$, we define $H(M)$ to be the hyperplane in $\mathbb R^n$ perpendicular to $\undim M$:
\[
	H(M):=\{x\in\mathbb R^n\,|\, x\cdot \undim M=0\}.
\]
The \emph{semistability set} of $M$ is defined to the subset $D(M)\subseteq H(M)$ given by
\[
	D(M)=\{x\in H(M)\,|\, x\cdot\undim M'\le 0\text{ for all $M'\subseteq M$}\}.
\]
The \emph{interior} $int\,D(M)$ of $D(M)$ is defined to be the subset of all $x\in D(M)$ so that $x\cdot\undim M'<0$ for all nonzero proper submodules $M'\subsetneq M$. The \emph{boundary} of $D(M)$ is the complement: $\partial D(M):=D(M)-int\,D(M)$. 
\end{defn}

Given a linear stability function $Z(x)=a\cdot x+ib\cdot x$, the \emph{corresponding path} $\gamma_Z:\mathbb R\to \mathbb R^n$ is the linear path given by
\[
	\gamma_Z(t)=tb-a.
\]

\begin{thm}\label{thm: path for Z goes through D(M)}
Let $M$ be an object of $\cL=nmod\text-\nM$ or $mod\text-\Lambda$. Then $M$ is $Z$-semistable, resp. $Z$-stable, if and only if $\gamma_Z(t)\in D(M)$, resp. $int\,D(M)$, for some $t\in R$. Furthermore, in that case, $t=\mu_Z(M)$.
\end{thm}

\begin{proof}
We prove the $Z$-semistable statement. $\gamma_Z(t_0)$ lies in $H(M)$ if and only if $t_0=\mu_Z(M)$. For any $M'\subset M$ let $t_1=\mu_Z(M')$. Then $\gamma_Z(t_1)\cdot \undim M'=0$. So, 
\[
	\gamma_Z(t_0)\cdot \undim M'=(t_0-t_1)b\cdot \undim M'
\]
Since $b\in\mathbb R^n$ has positive coordinates, this quantity is $\le0$ for all $M'\subset M$, making $\gamma_Z(t_0)\in D(M)$, iff $t_1=\mu(M')> t_0=\mu(M)$ for all $M'\subset M$, i.e., $M$ is $Z$-semistable.
\end{proof}

We recall that an object in an additive category is called \emph{Schurian} if its endomorphism ring is a division algebra. Schurian modules are also called ``bricks''. The following corollary was first observed in a more general context in \cite{Rudakov}.

\begin{cor}
If $M$ is $Z$-stable then $M$ is Schurian.
\end{cor}

\begin{proof}
If $M$ has a nontrivial endomorphism with image $M'\subsetneq M$ then $D(M)$ is contained in $D(M')$. So, $int\,D(M)$ is empty. By the Theorem above, $M$ cannot be $Z$-stable.
\end{proof}


\subsection{Example of cyclic $A_3$ with three possible algebras}

We consider the difference between $nmod\text-\nM$ and $mod\text-\Lambda$ when $\Lambda=T\nM/J$ for various admissible ideals $J$. Basically, any nilpotent module $M\in nmod\text-\nM$ is a $\Lambda$-module for some $\Lambda=T\nM/J$. If $M$ is a $\Lambda$-module then so are all of its subquotient modules. Therefore, the set $D(M)$ will be the same. It will be the set of all $x\in \mathbb R^n$ so that $x\cdot\undim M=0$ and $x\cdot\undim M'\le0$ for all $\Lambda$-submodules $M'$ of $M$. We give an example where $\cM$ is the simply laced quiver $Q$ given by a single oriented cycle of length 3. 

Recall that the modulated quiver given by a directed graph $Q$ has division algebras all equal to the ground field $K$ and bimodules $M_{ij}$ equal to $K^{m_{ij}}$ where $m_{ij}$ is the number of arrows from $i$ to $j$. The tensor algebra is called the \emph{path algebra} of $Q$ and denoted $KQ$. As a vector space it has a basis given by the paths including those of length zero which are the vertices of $Q$.
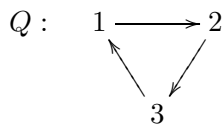
\[
\xymatrixrowsep{20pt}\xymatrixcolsep{10pt}
\xymatrix{
Q:&1 \ar[rr] && 2\ar[dl]\\
&& 3\ar[ul]
	}
\]
We consider the algebra $\Lambda_k=KQ/R^{k+1}$ where $R=RQ$ is the ideal generated by all paths of length $\ge1$. These are string algebras and all indecomposable modules are string modules \cite{BR}. Thus the algebra $\Lambda_k$ has $3(k+1)$ indecomposable modules (up to isomorphism) given by paths of length $\le k$ starting at any vertex.

In Figure \ref{FigureA3}, the left side shows the semistability sets for all six modules over $\Lambda_1=KQ/R^2$ and the right side shows all nine modules over $\Lambda_2=KQ/R^3$ and all 12 modules over $\Lambda_3=KQ/R^4$. 

Any $\Lambda_k$-module is a module over $\Lambda_j$ for all $j>k$. Therefore, the set $\bigcup D(M)$ for $\Lambda_1$ is a subset of $\bigcup D(M)$ for $\Lambda_2$. It will turn out that $\bigcup D(M)$ for $\Lambda_3$ is equal to that of $\Lambda_2$ which is the reason that the same figure (the right side of Figure \ref{FigureA3}) illustrates both. The figure shows the stereographic projection of the intersections of $D(M)$'s with the unit sphere $S^2$. 

For $\Lambda_1=KQ/R^2$, $D(S_i)=H(S_i)$ are the hyperplanes perpendicular to the unit vectors $\undim S_i=e_i$. These intersect $S^2$ in great circles which stereographically project to three circles in $\mathbb R^2$. Each set $D(P_i)=D(X_i)$ lies inside the $D(S_i)$ circle and outside the $D(S_{i+1})$ circle (with index modulo 3) since $S_i$ is a quotient of $P_i$ and $S_{i+1}$ is a submodule.

For $\Lambda_2=KQ/R^3$, we have projective modules $P_i'$ of length 3. The sets $D(P_i')$ all lie in the hyperplane $(1,1,1)^\perp$. But $D(P_1')$ lies inside $D(S_1)$ and outside $D(S_3)$ since $S_3\subset P_1'$ and $S_1=P_1'/X_2$. So, $D(P_1')$ is the part of the red circle between points $x$ and $z$. Similarly $D(P_2')$ is the part of the red circle from $x$ to $y$ and $D(P_3')$ is the part from $y$ to $z$. 

For $\Lambda_3=KQ/R^4$, the projective modules $P_i''$ have length 4 and they have the same simple $S_i$ on the top and bottom. This forces $D(P_i'')$ to lie in the line $H(P_i'')\cap H(S_i)=(1,1,1)^\perp\cap e_i^\perp$. Also, $X_i$ is a quotient of $P_i''$, so $D(P_i'')$ lies on the positive side of the hyperplane $D(X_i)$. In the figure we get the single points $x$ for $D(P_1'')$, $y$ for $D(P_2'')$ and $z$ for $D(P_3'')$.

{
\begin{center}
\begin{figure}[htbp]
\begin{tikzpicture} 
{
\clip (-3.5,-3.3) rectangle (11.8,4);
\begin{scope}[scale=.6]
\draw (-4.8,4.5) node{ $D(S_1)$};
\draw[color=blue] (-3.7,3) node{ $D(P_1)$};
\draw (4.8,4.5) node{ $D(S_2)$};
\draw[color=blue] (2.8,4.2) node{ $D(P_2)$};
\draw (3.2,-2) node{ $D(S_3)$};
\draw[color=blue] (.3,-2) node{ $D(P_3)$};
{
\begin{scope}[xshift=.75cm,yshift=1.3cm]
	\begin{scope}[rotate=60]
		\clip (0,-5) rectangle (5,5);
		\draw[color=blue,very thick] (0,0) ellipse [x radius=3cm,y radius=2.6cm];
	\end{scope}
		\draw[very thick] (.75,1.3) circle [radius=3cm]; 
		\draw[very thick] (-2.25,1.3) circle [radius=3cm]; 
		\draw[very thick] (-.75,-1.3) circle [radius=3cm]; 
	\begin{scope}
		\clip (-.75,-5.2) rectangle (-4.25,5);
		\draw[very thick,color=blue] (-.75,1.3) ellipse [x radius=3cm,y radius=2.6cm];
	\end{scope}
\end{scope}
}
	\begin{scope}[xshift=-.75cm,yshift=1.3cm,rotate=120]
		\clip (0,-5) rectangle (-5,5);
		\draw[very thick, color=blue] (0,0) ellipse [x radius=3cm,y radius=2.6cm];
	\end{scope}
\draw (0,-4.6) node{$6\ D(M)$'s for $\Lambda_1=KQ/R^2$};
\end{scope} 
}
{
\begin{scope}[scale=.6,xshift=13cm]
\draw (0,-4.6) node{$9\ D(M)$'s for $\Lambda_2=KQ/R^3$};
\coordinate (P1) at (-1.5,2.6);
\coordinate (P2) at (1.5,2.6);
\coordinate (P3) at (0,0);
\coordinate (nP1) at (3,0);
\coordinate (nP2) at (-3,0);
\coordinate (nP3) at (0,5.2);
\coordinate (I3) at (2.67,1.39);
\coordinate (nI2) at (-1.05,4.2);
\coordinate (nI1) at (-1.65,-.4);
\coordinate (S2) at (2.85,3.4);
\coordinate (nS2) at (-1.95,.6);
\draw (-4.8,.5) node{ $D(S_1)$};
\draw[color=red] (-4,-1.5) node{ $D(P_3')$};
\draw[color=red,thick,->] (-3.1,-1.5) ..controls (-2,-1.5) and (-1.2,-1.2)..(-1,-.9);
\draw[color=red] (-5.3,4.8) node{ $D(P_1')$};
\draw[color=red,thick,->] (-4.4,4.7) ..controls (-3.5,4.7) and (-3.5,3.5)..(-2.4,3.2);
\draw[color=blue] (-3.7,2.5) node{ $D(X_1)$};
\draw (4.8,4.5) node{ $D(S_2)$};
\draw[color=red] (5.6,1.5) node{ $D(P_2')$};
\draw[color=red,thick,->] (4.7,1.6) ..controls (4,1.6) and (4,2.5)..(2.7,2.5);
\draw[color=blue] (2.8,4.2) node{ $D(X_2)$};
\draw (3.2,-2) node{ $D(S_3)$};
\draw[color=blue] (.3,-2) node{ $D(X_3)$};
{
\begin{scope}[xshift=.75cm,yshift=1.3cm]
	\begin{scope}[rotate=60]
		\clip (0,-5) rectangle (5,5);
		\draw[color=blue,very thick] (0,0) ellipse [x radius=3cm,y radius=2.6cm];
	\end{scope}
		\draw[very thick] (.75,1.3) circle [radius=3cm]; 
		\draw[very thick] (-2.25,1.3) circle [radius=3cm]; 
		\draw[very thick] (-.75,-1.3) circle [radius=3cm]; 
	\begin{scope}
		\clip (-.75,-5.2) rectangle (-4.25,5);
		\draw[very thick,color=blue] (-.75,1.3) ellipse [x radius=3cm,y radius=2.6cm];
	\end{scope}
	\begin{scope}
		\draw[color=red,very thick] (-.75,.43)  circle [radius=2.68cm];
		\end{scope}
\end{scope}
}
	\begin{scope}[xshift=-.75cm,yshift=1.3cm,rotate=120]
		\clip (0,-5) rectangle (-5,5);
		\draw[very thick, color=blue] (0,0) ellipse [x radius=3cm,y radius=2.6cm];
	\end{scope}
\coordinate (I1) at (1.05,4.2);
\coordinate (I11) at (1.2,4.2);
\coordinate (I2) at (1.65,-.4);
\coordinate (I22) at (1.75,-.4);
\coordinate (nI3) at (-2.67,1.39);
\foreach \x in {I1,I2,nI3}
	\draw[fill] (\x) circle(3pt);
	\draw (I11) node[above]{$x$};
	\draw (I22) node[below]{$y$};
	\draw (nI3) node[left]{$z$};
\end{scope} 
}
\end{tikzpicture}
\caption{Stereographic projections of the intersections of $D(M)$'s with the unit sphere $S^2\subset \mathbb R^3$. The modules $X_i$ over $\Lambda_2$ are equal to the projective $\Lambda_1$-modules $P_i$. The projective $\Lambda_2$-modules $P_i'$ share identical dimension vectors $\undim P_i'=(1,1,1)$. And, the union of $D(P_i')$, $i=1,2,3$, form the hyperplane $(1,1,1)^\perp$ which is the red circle in the figure.
}
\label{FigureA3}
\end{figure}
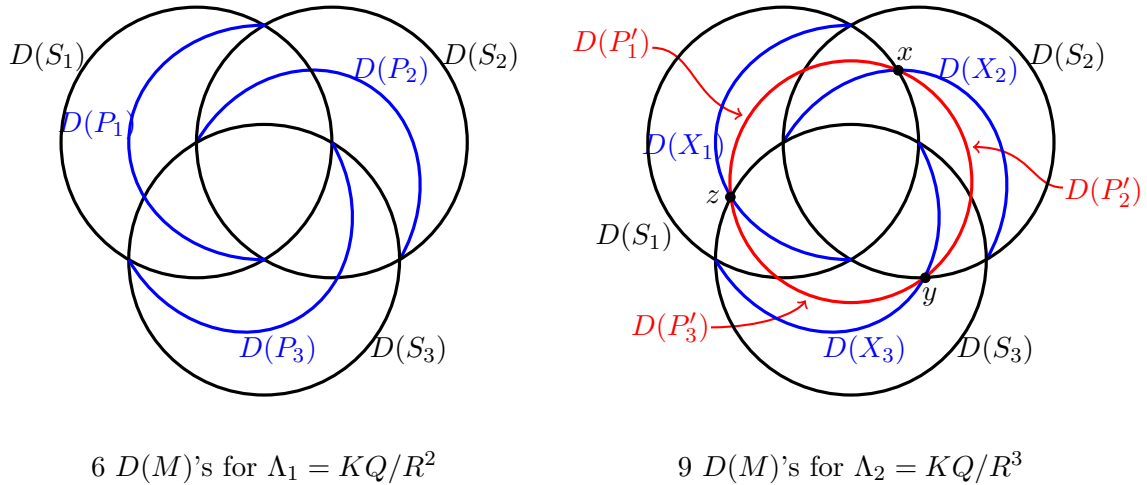
\end{center}
}
%

%
%

\section{Nonlinear stability functions}\label{sec2}

In this section we define nonlinear stability functions $Z_t:K_0\cL\to \mathbb C$, $t\in\mathbb R$, and show that the corresponding nonlinear paths $\gamma_Z:\mathbb R\to \mathbb R^n$ cross the semistability sets $D(M)$ of the $Z_\bullet$-semistable modules. When $Z_\bullet$ is ``green'' we obtain a Harder-Narasimhan stratification of the category $\cL=mod\text-\Lambda$ or $nmod\text-\cM$.


\subsection{Definition of nonlinear $Z_\bullet$}

\begin{defn}\label{def: nonlinear stability function} We define a \emph{nonlinear stability function}\footnote{This is the same as the ``stability function'' considered in \cite{BST2}.} 
 on the $\cL$ to be a smooth ($C^1$) family of linear functions

\[
	Z_t:K_0\cL\cong \mathbb Z^n \to \mathbb C, \quad t\in\mathbb R
\]
given by
\[
	Z_t(x)=a_t\cdot x+b_t\cdot x\sqrt{-1}=r_t(x)e^{i\theta_t(x)}
\]
where:
\begin{enumerate}
\item $a_t\in\mathbb R^n$, $b_t\in (0,\infty)^n$, $0<\theta_t(x)<\pi$, i.e. the coordinates of $b_t$ are positive for all $t$.
\item $a_t$ and $b_t$ have velocity 0 for $|t|$ large, giving four constants: $a_\infty,b_\infty,a_{-\infty}$ and $b_{-\infty}$. 
\end{enumerate}
Note that $Z_t$ may be a different function for each $t$. Thus, a nonlinear stability function could also be called a ``time-dependent central charge''. 

For any object $M\in \cL$ let 
\begin{equation}\label{eq: def of slope}
	\mu_t(M):=\cot \theta_t(M)=\frac{a_t\cdot \undim M}{b_t\cdot \undim M}
\end{equation}
where $\theta_t(M)$ is short for $\theta_t(\undim M)$.
\end{defn}

\begin{defn}
Let $Z_\bullet$ be a nonlinear stability function on $\cL$. We say that $M\in\cL$ is \emph{$Z_\bullet$-stable/semistable} if, for some $t_0\in\mathbb R$ we have the following.
\begin{enumerate}
\item $M$ is stable/semistable with respect to $Z_{t_0}$.
\item $\mu_{t_0}(M)=t_0$.
\end{enumerate}
The pair $(M,t_0)$ will be called a $Z_\bullet$-\emph{stable/semistable pair}. Such a pair $(M,t_0)$ is \emph{green} or \emph{red} if
\begin{equation}\label{eq: when (b,t) is Z-green}
	\left.\frac d{dt}\mu_t(M)\right|_{t=t_0}<1
\end{equation}
or $>1$, respectively. We say the pair $(M,t_0)$ is \emph{generic} if it is green or red, i.e., the expression \eqref{eq: when (b,t) is Z-green} is not equal to 1. We say that $Z_\bullet$ is \emph{green}, resp. \emph{generic}, if all $Z_\bullet$-semistable pairs are green, resp. generic.
\end{defn}

All linear stability functions are green since $0<1$. For nonlinear stability functions, the condition \eqref{eq: when (b,t) is Z-green} is needed to obtain a ``green path'' and as a consequence obtain the Harder-Harasimhan filtration for any module $M$.


\subsection{Comparison with corresponding path}\label{ss: green paths}

Given any nonlinear stability function \[
Z_t(x)=a_t\cdot x+b_t\cdot x\sqrt{-1},\]
let $\gamma_Z:\mathbb R\to \mathbb R^n$ be the path
\[
	\gamma_Z(t)=t b_t- a_t
\]
Note that $\gamma_Z(t)\in H(M)$ if and only if 
\[
\gamma_Z(t)\cdot \undim M=tb_t\cdot \undim M-a_t\cdot \undim M=0
\] 
which is equivalent to:
\[
	t=\frac{a_t\cdot \undim M}{b_t\cdot \undim M}=\mu_t(M).
\]

\begin{lem}\label{lem1}\begin{enumerate}
\item $\gamma_Z(t_0)\in D(M)$ if and only if $M$ is $Z_{t_0}$-semistable. 
\item
$\gamma_Z(t_0)\in int\,D(M)$ if and only if $M$ is $Z_{t_0}$-stable. And, in that case, $M$ is Schurian.
\end{enumerate}
\end{lem}

\begin{proof} This follows from Theorem \ref{thm: path for Z goes through D(M)} applied to the linear stability function $Z_{t_0}$.
\end{proof}

\begin{defn}\label{def of red or green path}
Let $v(t)=d\gamma_Z(t)/dt$ be the velocity vector of the smooth path $\gamma_Z$ at time $t$. If the path $\gamma_Z(t)$ crosses $D(M)\subset H(M)$ at $t=t_0$, we say that the crossing is \emph{green} if $v(t_0)\cdot \undim M>0$. We say it is \emph{red} if $v(t_0)\cdot \undim M<0$. Thus any transverse intersection is either green or red.
\end{defn}

\begin{lem}\label{lem2}
Suppose that $\gamma(t_0)\in D(M)$. Then the derivative of $\mu_t(M)$ at $t=t_0$ is not equal to $1$ if and only if $\gamma$ crosses the hyperplane $H(M)$ transversely at $t=t_0$. Furthermore the crossing is green/red if and only if the sign of $d\mu_t(M)/dt|_{t=t_0}-1$ is negative/positive, respectively.
\end{lem}

\begin{proof}
The statement is that the following have the same sign for $\beta=\undim M$.
\begin{equation}\label{eq: first expression for green or red}
	1-\frac d{dt}\left.\mu_t(\beta)\right|_{t=t_0}=1-\frac{(a_{t_0}'\cdot\beta)(b_{t_0}\cdot\beta)-(a_{t_0}\cdot\beta)(b_{t_0}'\cdot\beta)}{(b_{t_0}\cdot\beta)^2}
\end{equation}
\begin{equation}\label{eq: second expression for green or red}
	v(t_0)\cdot \beta=\frac d{dt}\left.\gamma_Z(t)\cdot\beta\right|_{t=t_0}=b_{t_0}\cdot \beta+t_0b_{t_0}'\cdot\beta -a_{t_0}'\cdot\beta
\end{equation} 
{Substituting $t_0=\mu_{t_0}(M)=\frac{a_{t_0}\cdot\beta}{b_{t_0}\cdot\beta}$ in \eqref{eq: second expression for green or red}, {we see that this second expression is equal to}
\[
	{b_{t_0}\cdot\beta}+ \frac{a_{t_0}\cdot\beta}{b_{t_0}\cdot\beta}b_{t_0}'\cdot\beta-a_{t_0}'
	\cdot\beta={b_{t_0}\cdot\beta}+ 
	\frac{(a_{t_0}\cdot \beta)(b_{t_0}'\cdot\beta)}{b_{t_0}\cdot\beta}-
	\frac{(a_{t_0}'\cdot \beta)(b_{t_0}\cdot\beta)}{b_{t_0}\cdot\beta}
\]
 which is \eqref{eq: first expression for green or red} times $b_{t_0}\cdot\beta=b_{t_0}\cdot\undim M>0$. So, the two expressions have the same sign.}
\end{proof}

\begin{lem}\label{lem3}
The coordinates of $\gamma_Z(t)$ are all negative for $t<<0$ and they are all positive for $t>>0$.
\end{lem}

\begin{proof}
When $|t|$ is very large, $\gamma_Z(t)=tb_t-a_t$ is dominated by the term $tb_t$ whose coordinates are nonzero with the same sign as $t$.
\end{proof}

The important properties of $\gamma_Z$ are summarized by the following definition and theorem.

\begin{defn}\label{def: reddening path}
A \emph{reddening path} for $\cL=nmod\text-\nM$ or $mod\text-\Lambda$ is defined to be a $C^1$ path $\gamma:\mathbb R\to\mathbb R^n$ with the following two properties.
\begin{enumerate}
\item All coordinates of $\gamma(t)$ are negative for $t<<0$ and positive for $t>>0$.
\item Whenever $\gamma$ crosses a semistability set $D(M)$ for $M\in\cL$, the crossing is transverse, i.e., $D(M)$ has codimension 1 and the curve $\gamma$ is not tangent to the hyperplane $H(M)$.
\end{enumerate}
A reddening path will be called a \emph{green path} for $\Lambda$ if all crossings are green.
\end{defn}

\begin{thm}\label{thm: (a) iff (b)}
Let $Z_\bullet$ be a nonlinear stability function for $\cL$. Then $Z_\bullet$ is generic, resp. green, if and only if the corresponding path $\gamma_Z$ is a reddening path, resp. a green path. Conversely, for any reddening path $\gamma$ there is a nonlinear $Z_\bullet$ so that $\gamma(t)\in D(M)$ if and only if $M$ is $Z_t$-semistable, equivalently, if $\gamma_Z(t)\in D(M)$.
\end{thm}

\begin{proof}
The relation between $Z_\bullet$ and $\gamma_Z$ is proved in Lemma \ref{lem2}.

To prove the second statement, let $\gamma$ be any reddening path. Since each $H(M)$ and thus each $D(M)$ contains no points whose coordinates are all positive or all negative, the two tails of an arbitrary reddening path $\gamma$ can be modified to be stationary so that $\gamma$ becomes equal to some $\gamma_Z$ without changing when, where or with what velocity it meets any semistability set $D(M)$ for $M\in mod$-$\Lambda$. The second statement follows.
\end{proof}


\subsection{Wide subcategory}\label{ss: wide subcategory proof}

Recall that a \emph{wide subcategory} of an abelian category $\cA$ is a full subcategory $\cW$ which is closed under direct summands, extensions, kernels and cokernels. We observe that, if $\cB$ is an exactly embedded abelian full subcategory of $\cA$, then $\cW\cap \cB$ is a wide subcategory of $\cB$.


\begin{defn}
For any subset $S$ of $\mathbb R^n$ let $\cW(S)$ be the full subcategory of $\cL=nmod\text-\nM$ or $mod\text-\Lambda$ of all objects $M$ so that $S\subseteq D(M)$.
\end{defn}

It is well-known and easy to see that $\cW(S)$ is a \emph{wide subcategory} of $\cL$. We review the proof in our setting using  equivalent statements about the sets $D(M)$.

\begin{lem}\label{lem: W(S) is closed under ext, ker, coker}
For any $S\subseteq\mathbb R^n$, $\cW(S)$ is closed under extensions, cokernels of monomorphisms, kernels of epimorphisms. Equivalently,
\[
	D(A)\cap D(B)\subseteq D(E)
\]
\[
	D(A)\cap D(E)\subseteq D(B)
\]
\[
	D(E)\cap D(B)\subseteq D(A)
\]
for any short exact sequence $A\to E\to B$.
\end{lem}

\begin{proof}
By definition of $\cW(S)$, $A,B\in\cW(S)$ iff $S\subseteq D(A)\cap D(B)$. We want to prove that $E\in \cW(S)$ which is equivalent to $S\subseteq D(E)$. So, the two formulations of first statement in the lemma are equivalent. We prove the second. Let $x\in D(A)\cap D(B)$. We want to show that $x\in D(E)$. The first step is easy:
\[
	x\cdot\undim E=x\cdot\undim A+x\cdot \undim B=0
\]
For any $C\subseteq E$, we have a short exact sequence $A\cap C\to C\to D\subseteq B$. So,
\[
	x\cdot\undim C=x\cdot \undim (A\cap C) + x\cdot\undim D\le 0.
\]
So, $x\in D(E)$ proving that $\cW(S)$ is closed under extensions. The proofs of the other two statements are similar.
\end{proof}

\begin{lem}\label{lem: W(S) is closed under images}
Let $f:A\to B$ be a morphism in $\cL$ with image $C$. Then $D(A)\cap D(B)\subseteq D(C)$. Equivalently, $\cW(S)$ is closed under images.
\end{lem}

\begin{proof}
Let $x\in D(A)\cap D(B)$. Since $C\subseteq B$ we have $x\cdot \undim C'\le 0$ for any $C'\subseteq C$. Since $C$ is a quotient module of $A$ we have $x\cdot \undim C\ge 0$. Therefore $x\in D(C)$.
\end{proof}

\begin{lem}\label{lem: W(S) is closed under summands}
$D(A\oplus B)=D(A)\cap D(B)$. So, $\cW(S)$ is closed under summands.
\end{lem}

\begin{proof}
By Lemma \ref{lem: W(S) is closed under ext, ker, coker}, $D(A)\cap D(B)\subseteq D(A\oplus B)$. By Lemma \ref{lem: W(S) is closed under images}, $D(A\oplus B)\subseteq D(A)$ (and $D(A\oplus B)\subseteq D(B)$) since $A$ is the image of an endomorphism of $A\oplus B$.
\end{proof}

\begin{thm}
For any $S\subseteq \mathbb R^n$, $\cW(S)$ is a wide subcategory of $\cL$.
\end{thm}

\begin{proof}
Any summand of $M\in \cW(S)$ lies in $\cW(S)$ by Lemma \ref{lem: W(S) is closed under summands}. $\cW(S)$ is closed under extensions by Lemma \ref{lem: W(S) is closed under ext, ker, coker}. Given any morphism $f:A\to B$ where $A,B\in \cW(S)$, the exact sequences $\ker f\hookrightarrow A\twoheadrightarrow \text{im}\,f$ and $\text{im}\,f\hookrightarrow B\twoheadrightarrow\coker f$ show that $\ker f,\coker f\in \cW(S)$ using Lemmas \ref{lem: W(S) is closed under ext, ker, coker} since $\text{im}\,f\in \cW(S)$ by Lemma \ref{lem: W(S) is closed under images}.
\end{proof}


\subsection{Harder-Harasimhan filtration}\label{ss: HN filtration}

We now come to the purpose of the sign condition \eqref{eq: when (b,t) is Z-green} on the nonlinear stability function $Z_\bullet$. Namely, this condition will imply that the walls $D(M)$ crossed by the corresponding green path $\gamma_Z$ give a Harder-Narasimhan filtration for any representation. The detailed proof that follows is based on the half-page proof in \cite{R}. The proof in \cite{B}, Proposition 2.4, is also very short and elegant. The original proof comes from \cite{Rudakov}. For just the idea of the proof, the reader should go to these original sources.

\begin{lem}\label{lem: stable means in the wide subcat}
{Given a nonlinear stability function $Z=Z_\bullet$ and $t_0\in\mathbb R$ let $\cW_Z(t_0)$ be the full subcategory of $\cL=nmod\text-\nM$ or $mod\text-\Lambda$ consisting of all objects $M\in\cL$ so that $(M,t_0)$ is a $Z$-semistable pair. Then $\cW_Z(t_0)$ is a wide subcategory of $\cL$.}
\end{lem}

\begin{proof} By Theorem \ref{thm: (a) iff (b)}, the pair $(W,t_0)$ is $Z$-semistable if and only if $Z_{t_0}(t_0)\in D(M)$. Therefore $\cW_Z(t_0)=\cW(Z_{t_0}(t_0))$ is a wide subcategory.
\end{proof}

{Given a nonlinear stability function $Z_\bullet$ which is green, we will show that the family of wide subcategories $\cW_Z(t), t\in \mathbb R$ gives an HN-stratification of $\cL$ as defined below.

\begin{defn}\label{def: HN-stratification}
A \emph{Harder-Narasimhan (HN)-stratification} of $\cL$ (also known as an \emph{HN-system}) is defined to be any family of wide subcategories $\{\cW_t\},t\in\mathbb R$, in $\cL$ satisfying the condition that, for any object $M$, there exists a unique finite sequence $t_1<t_2<\cdots<t_m$ and a unique filtration
\[
	0=M_0\subsetneq M_1\subsetneq \cdots\subsetneq M_m=M
\]
having the property that $M_k/M_{k-1}\in\cW_{t_k}$ for all $k$. This filtration is called the \emph{Harder-Narasimhan (HN)-filtration} of $M$.
\end{defn}

\begin{rem}\label{rem: comparing HS-stratifications}
We note that, given an HN-stratification $\{\cW_t\}$ of $nmod\text-\nM$, we obtain an HN-stratification $\{\cW_t'\}$ of $mod\text-\Lambda$ for any $\Lambda=T\cM/J$ where $\cW_t'=\cW_t\cap mod\text-\Lambda$, the full subcategory of $\cW_t$ with objects $M$ which are annihilated by $J$.
\end{rem}

The concept of an HN-stratification comes from \cite{B} where it is called a ``slicing'' of the category. The concept of HN-filtration comes from \cite{HN}.}

\begin{defn}\label{def: t0 and t1}
For any object $M\in\cL$, the set of values of $\mu_t(M)$ (defined in \ref{eq: def of slope}) is bounded. So, $t>\mu_t(M)$ for $t>>0$ and $t<\mu_t(M)$ for $t<<0$. So, the set of all $t\in\mathbb R$ for which $\mu_t(M)=t$ is closed, bounded and nonempty. Let $t_0(M)$ be the smallest element of this set. Since $t_0(M)$ depends only on the dimension vector of $M$, there are only finitely many values of $t_0(M')$ for all submodules $M'\subseteq M$. Let $t_1(M)$ be the smallest value of $t_0(M')$ for all nonzero $M'\subseteq M$.
\end{defn}

\begin{lem}\label{lem: when sigma is below t}
If $\mu_t(M)<t$ then $t_0(M)<t$.
\end{lem}

\begin{proof}
Since $\mu_{t'}(M)>t'$ for $t'<<0$, $\exists t''<t$ so that $\mu_{t''}(M)=t''\ge t_0(M)$.
\end{proof}

\begin{lem}\label{lem: sigma is above t}
For all $M'\subseteq M$ and $t<t_1(M)$ we have
\[
	\mu_t(M')>t.
\]
By continuity, we also have $\mu_{t_1(M)}(M')\ge t_1(M)$.
\end{lem}

\begin{proof}
If $\mu_t(M')\le t$ then, since $|\mu_t(M')|$ is bounded, there must be some $t'\le t$ so that $\mu_{t'}(M')=t'$. Then $t_1(M)\le t'\le t$ by definition of $t_1(M)$, proving the lemma.
\end{proof}

\begin{prop}\label{prop: when is M Z-semistable}
Suppose $Z_\bullet$ is green for $\cL$ and $M\in \cL$ is $Z_t$-semistable for some $t\in \mathbb R$. Then 
\[
	t=t_0(M)=t_1(M)
\]
Conversely, if $t_0(M)=t_1(M)$, then $M$ is $Z_{t_0}$-semistable for $t_0=t_0(M)$.
\end{prop}

\begin{proof}
Suppose $t_0=t_0(M)=t_1(M)$. Then $\mu_{t_0}(M)=t_0$ and, by Lemma \ref{lem: sigma is above t} above, $\mu_{t_0}(M')\ge t_0$ for all $M'\subseteq M$. So, $M$ is $Z_{t_0}$-semistable.

Conversely, suppose that $M$ is $Z_t$-semistable and $Z_\bullet$ is green. Then \[
t\ge t_0(M)\ge t_1(M).\]
So, it suffices to show that $t\le t_1(M)$. To prove this, suppose not. Then $
t>t_1=t_1(M)$.
By definition of $t_1(M)$ there is $0\neq M'\subseteq M$ so that $\mu_{t_1}(M')=t_1<t$. 

Let $M'\subseteq M$ be minimal with the property that $
	\mu_{t'}(M')=t'$ for some $t'<t$. Taking $t'$ minimal we may assume $t'=t_0(M')<t$. By minimality of $M'$ we have $t_0(M'')\ge t>t'$ for all $M''\subsetneq M'$. By Lemma \ref{lem: sigma is above t} this implies $\mu_{t'}(M'')> t'$. So, $M'$ is $Z_{t'}$-stable. Since $Z_\bullet$ is \emph{green for $\Lambda$}, any $t''>t'$ sufficiently close to $t'$ has the property that $\mu_{t''}(M')<t''$. Since $t'<t$, we can also take $t''<t$. But $\mu_t(M'')\ge t$ by the assumption that $M$ is $Z_t$-semistable. By the intermediate value theorem in calculus, there must be a point $t_\ast$ with $t''<t_\ast\le t$ so that $\mu_{t_\ast}(M'')=t_\ast$. Taking $t_\ast$ minimal we must have:
\[
	\left.\frac d{dt}\mu_t(M'')\right|_{t=t_\ast}\ge1
\]
So, $M''$ cannot be $Z_{t_\ast}$-semistable. So, there is some $M_\ast\subsetneq M''$ so that $\mu_{t_\ast}(M_\ast)<t_\ast$. By Lemma \ref{lem: when sigma is below t} there is some $t_\ast'<t_\ast$ so that $\mu_{t_\ast'}(M_\ast)=t_\ast'<t_\ast<t$ contradicting the minimality of $M'$. So, $t\le t_0\le t_1\le t$, showing that the three numbers are equal.
\end{proof}

\begin{rem}
For $Z_\bullet$ green, we say $M$ is \emph{$Z_\bullet$-stable/semistable} if $M$ is $Z_t$-stable/semistable for some $t$. Proposition \ref{prop: when is M Z-semistable} implies that the value of $t$ is uniquely determined. It also implies that $M$ is $Z_\bullet$-semistable if and only if $t_0(M)=t_1(M)$.
\end{rem}

\begin{lem}\label{lem: M1 is unique}
Let $M_1\subseteq M$ be maximal so that $t_0(M_1)=t_1(M)$. Then $M_1$ contains all $M'\subseteq M$ having $t_0(M')=t_1(M)$. In particular, $M_1$ is unique.
\end{lem}

\begin{proof}
Suppose $M_1$ does not have this property. Then, there is $M_1'\subsetneq M$ with $t_0(M_1')=t_1$ so that $M_1,M_1'$ do not contain each other. Also, $M_1+M_1'$ properly contains $M_1$. So, $t_0(M_1+M_1')>t_1$ by maximality of $M_1$. This implies $\mu_{t_1}(M_1+M_1')>t_1$. Since $\mu_{t_1}(M_1)=t_1$ we must have 
\[
	\mu_{t_1}\left(\frac{M_1+M_1'}{M_1}\right)=\mu_{t_1}\left(\frac{M_1'}{M_1\cap M_1'}\right)>t_1.
\]
Since $\mu_{t_1}(M_1')=t_1$ we conclude that $\mu_{t_1}(M_1\cap M_1')<t_1$ contradicting the definition of $t_1=t_1(M)$. This proves the lemma.
\end{proof}

\begin{lem}\label{lem: t2>t1}
{Let $Z_\bullet$ be green for $\cL$, let $M\in\cL$ and} let $M_1\subseteq M$ be the maximal submodule so that $t_0(M_1)=t_1=t_1(M)$. If $M_1\neq M$ then $t_1(M/M_1)>t_1$.
\end{lem}

\begin{proof}
Suppose not. Then the set
\[
	S=\{s\le t_1\,|\, \mu_s(M'/M_1)=s \text{ for some }M_1\subsetneq M'\subseteq M\}
\]
is closed and nonempty. Also, $t_1\notin S$ since that would imply $\mu_{t_1}(M')=t_1$ contradicting the maximality of $M_1$. Let $m<t_1$ be the maximal element of $S$ and let $M''\supsetneq M_1$ be minimal so that $\mu_m(M'/M_1)=m$.\vskip.2cm

\underline{Claim}: $M'/M_1$ is $Z_m$-semistable.\vskip.2cm

Proof: Suppose not. Then there exists $M_1\subsetneq M''\subsetneq M'$ so that $\mu_m(M''/M_1)<m$. But, as we observed in the proof of Lemma \ref{lem: M1 is unique}, $\mu_{t_1}(M''/M_1)>t_1$ by maximality of $M_1$. This implies that $\mu_s(M''/M_1)=s$ for some $m<s<t_1$ contradicting the maximality of $m$. So, the Claim holds.\vskip.2cm

Since $Z_\bullet$ is green, we have $\mu_s(M'/M_1)<s$ for $s>m$ close to $m$, in particular, for some $m<s<t_1$. But $\mu_{t_1}(M'/M_1)>t_1$. This leads to the same contradiction as in the proof of the Claim above. Therefore, the Lemma holds.
\end{proof}

We are now ready to prove the main theorem about green nonlinear stability functions.

\begin{thm}[HN-filtration]\label{thm: HN filtration}
Let $Z_\bullet$ be a nonlinear stability function for {$\cL=nmod\text-\nM$ or $mod\text-\Lambda$ which is green and let $M\in \cL$.} Then there exist a unique sequence of real numbers $t_1<t_2<\cdots<t_m$ and a unique filtration
\[
	0=M_0\subsetneq M_1\subsetneq M_2\subsetneq\cdots\subsetneq M_m=M
\]
with the property that $M_i/M_{i-1}$ is $Z_{t_i}$-semistable for $i=1,2,\cdots,m$. Furthermore, $t_1=t_1(M)$ as given in Definition \ref{def: t0 and t1}.
\end{thm}

By Lemma \ref{lem: stable means in the wide subcat}, Theorem \ref{thm: HN filtration} is equivalent to the statement that $\cW_Z(t),t\in\mathbb R$, is an HN-stratification of $mod\text-\Lambda$. The sequence of submodules $M_1,M_2,\cdots,M_m$ will be called the \emph{Harder-Narasimhan (HN) filtration} of $M$ with respect to $Z_\bullet$.

\begin{proof} We can assume that $\cL=mod\text-\Lambda$ since, in the case $\cL=nmod\text-\nM$, we can take $\cL=mod\text-\Lambda$ where $\Lambda=T\nM/J$, $J$ begin the annihilator of $M$. Given that $Z_\bullet$ satisfies the required stability conditions for all nilpotent representations of $\nM$, a fortiori, it satisfies these conditions for all $\Lambda$-modules. So, $Z_\bullet$ is a green nonlinear stability function for $mod\text-\Lambda$.

(Existence) We first show the existence of the sequence $(M_1,t_1),\cdots,(M_m,t_m)$ by induction on the size of $M\in mod$-$\Lambda$. If $M$ is simple then we let $t_1=t_0(M)=t_1(M)$. Then $M_1=M$ is $Z_{t_1}$-stable and we are done.

If $M$ is not simple, let $t_1=t_1(M)$ be as given in Definition \ref{def: t0 and t1}. Let $M_1$ be the unique largest submodule of $M$ with $t_0(M_1)=t_1(M)$. Then $M_1$ is $Z_{t_1}$-semistable by Proposition \ref{prop: when is M Z-semistable}. By induction on the size of $M$, the module $M/M_1$ has a unique HN-filtration
\[
	0=M_1/M_1\subsetneq M_2/M_1\subsetneq \cdots\subsetneq M_m/M_1=M/M_1
\]
and there are $t_2<t_3<\cdots<t_m$ so that $(M_k/M_1)/(M_{k-1}/M_1)\cong M_k/M_{k-1}$ is $Z_{t_k}$-semistable for $k=2,3,\cdots,m$. Furthermore, $t_2=t_1(M/M_1)$ by induction and this is $>t_1$ by Lemma \ref{lem: t2>t1}. So, $M_1,M_2,\cdots,M_m$ is an HN-filtration of $M$.

(Uniqueness) If $(M_1',t_1')$ is the beginning of another HN-filtration of $M$ then the key step is to show that $t_1'=t_1$. This follows from the fact that
\[
	M_1/(M_1\cap M_1')\cong (M_1+M_1')/M_1'.
\]
Since the theorem holds for $M/M_1'$, we have $t_2'=t_1(M/M_1')$ and:
\[
	t_1\ge t_0(M_1/(M_1\cap M_1'))=t_0((M_1+M_1')/M_1')\ge t_1(M/M_1')=t_2'>t_1'
\]
By definition of $t_1(M)$ we have $t_1\le t_1'$. So, $t_1'=t_1$.

Next we show that $M_1'=M_1$. By Lemma \ref{lem: M1 is unique}, $M_1'\subseteq M_1$. If they are not equal then $M_1/M_1'$ is nonzero with $\mu_{t_1}(M_1/M_1')=t_1$. This contradicts the induction hypothesis that $t_2'=t_1(M/M_1')>t_1'=t_1$. So, $M_1'=M_1$. The rest of the filtration is the unique filtration of $M/M_1$. So, the HN-filtration of $M$ is unique.
\end{proof}

Theorem \ref{thm: HN filtration} has many well-known consequences and reformulations. 

\begin{cor}\label{cor: hom orthogonal}
If $A,B$ are $Z_\bullet$-semistable with $t_0(A)<t_0(B)$ then $\Hom(A,B)=0$.
\end{cor}

\begin{proof}
If $f:A\to B$ is nonzero, $A\subset A\oplus B$ and $A\cong (\text{graph of $f$}) \subset A\oplus B$ are two HN filtrations of $A\oplus B$ contradicting it uniqueness.
\end{proof}

Let $A\in \cW_Z(s)$, $B\in \cW_Z(t)$ where $s<t$ {(defined in Lemma \ref{lem: stable means in the wide subcat}). By Corollary \ref{cor: hom orthogonal}, $\Hom(A,B)=0$.} More generally we have the following well-known corollary where, for any connected subset $S\subseteq \mathbb R$, {we define $\cW_Z(S)$ to be the full subcategory of $\cL$ of all objects $M$ so that the numbers $t_1,\cdots,t_m$ for the HN-filtration of $M$ all lie in $S$.} We recall that a \emph{torsion pair} is a pair of full subcategories $(\cT,\cF)$ in an abelian category $\cA$ so that
\begin{enumerate}
\item An object $A\in \cA$ lies in $\cT$ if and only if $\Hom_\cA(A,B)=0$ for all $B\in\cF$.
\item An object $B\in \cA$ lies in $\cF$ if and only if $\Hom_\cA(A,B)=0$ for all $A\in\cT$.
\end{enumerate}

Important examples of $\cW_Z(S)$ are when $S$ is an interval.

\begin{cor} Let $Z_\bullet$ be green. Then, for any $t_0\in \mathbb R$,
$\cW_Z(-\infty,t_0]$ and $\cW_Z(t_0,\infty)$ form a torsion pair in $\cL=nmod\text-\cM$ or $mod\text-\Lambda$ as do $\cW_Z(-\infty,t_0)$ and $\cW_Z[t_0,\infty)$.
\end{cor}

\begin{proof} It is clear that $\Hom(A,B)=0$ for any $A\in \cW_Z(-\infty,t_0]$ and $B\in \cW_Z(t_0,\infty)$. Conversely,
$M\in \cW_Z(t_0,\infty)$ if and only if it does not have a submodule in $\cW_Z(-\infty,t_0]$ and $M\in \cW_Z(-\infty,t_0]$ iff it does not have a quotient module in $\cW_Z(t_0,\infty)$. So, $\cW_Z(-\infty,t_0]$ and $\cW_Z(t_0,\infty)$ form a torsion pair. The other case is similar.
\end{proof}


%
%

\section{Finite HN-systems and forward hom-orthogonality}\label{sec3}

In this section we consider stability functions and paths which are finite with respect to a fixed finite dimensional algebra $\Lambda$ in the sense that the path crosses only finitely many walls $D(M)$ given by $\Lambda$-modules $M$. The resulting HN-stratification of $mod\text-\Lambda$ is of course finite. We call it a ``finite HN-system'' for $\Lambda$ and we show that it is equivalent to a ``forward hom-orthogonal sequence''.

\subsection{Finite HN-systems}

For any $\Lambda$-module $M$ let $\cE(M)$ denote the full subcategory of $mod\text-\Lambda$ of all modules $X$ having a filtration with all subquotients isomorphic to $M$, i.e., $X$ is an ``iterated self-extension'' of $M$. The following is an easy exercise.

\begin{lem}
If $M$ is Schurian then $\cE(M)$ is an abelian category.\qed
\end{lem}

\begin{lem}\label{lem: W=E(M)}
Let $M$ be the unique indecomposable object of a wide subcategory $\cW$ of $mod\text-\Lambda$. Then $M$ is Schurian and $\cW=\cE(M)$.
\end{lem}

\begin{proof}
Since $\cW$ is closed under extensions and contains $M$, we have $\cE(M)\subseteq \cW$. Conversely, let $X$ be the minimal object of $\cW$ which is not in $\cE(M)$. Let $X_0\subset X$ be the smallest submodule of $X$ which lies in $\cW$. Then $X_0$ must be Schurian since the image of any nonzero endomorphism of $X_0$ must also lie in $\cW$. By assumption, $X_0= M$. By minimality of $X$ we have $X/M\in\cE(M)$. So, $X\in\cE(M)$. We conclude that $\cW=\cE(M)$.
\end{proof}

\begin{defn}\label{def: finite HN-system}
By a \emph{finite Harder-Narasimhan (HN) system} for $mod\text-\Lambda$ we mean a finite sequence of abelian subcategories $\cE(M_1),\cdots,\cE(M_m)\subset mod\text-\Lambda$ with $M_i$ Schurian which give an HN-stratification of $mod\text-\Lambda$. I.e., for any $\Lambda$-module $X$ there is a unique filtration
\[
	0=X_0\subseteq X_1\subseteq X_2\subseteq\cdots\subseteq X_m=X
\]
so that each $X_k/X_{k-1}$ is in $\cE(M_k)$.
\end{defn}

A finite HN-system is exactly the kind of HN-stratification that we get from a ``finite'' nonlinear stability function $Z_\bullet$. Bridgeland considered locally finite HN-stratifications or ``slicing'' in \cite{B} which are very similar to finite HN-systems.

\begin{defn}\label{def: finite Z and gamma}
A nonlinear stability function $Z_\bullet$ for $\Lambda$ will be called \emph{finite} if it is generic and there are only finitely many semistable pairs $(M_i,t_i)$, up to isomorphism, with indecomposable $M_i$ and the $t_i$ are all distinct.

A reddening path $\gamma:\mathbb R\to\mathbb R^n$ will be called \emph{finite} if there are only finitely many real numbers $t_i$ so that $\gamma(t_i)$ lie in some $D(M_i)$ for $M_i$ indecomposable and if, for each $t_i$, $M_i$ is uniquely determined up to isomorphism.
\end{defn}

The statement that, for the semistable pairs $(M_i,t_i)$ of a finite $Z_\bullet$, the $t_i$ are all distinct means that $M_i$ is uniquely determined by $t_i$. This implies the following.

\begin{prop}\label{prop: (a) iff (b)}
Let $Z_\bullet$ be a nonlinear stability function for $\Lambda$ with corresponding path $\gamma_Z$. Then $Z_\bullet$ is finite if and only if $\gamma_Z$ is finite.\qed
\end{prop}

\begin{thm}\label{thm: finite green (a) implies finite (c)}
Let $Z_\bullet$ be a finite, green stability function with semistable pairs $(M_i,t_i)$. Then each pair is $Z_\bullet$-stable, each $M_i$ is Schurian and $\cE(M_1),\cdots,\cE(M_m)$ form a finite HN-system.
\end{thm}

\begin{proof}
The pair $(M_i,t_i)$ is stable iff $\gamma_Z(t_i)\in int\,D(M_i)$. So, suppose not. Then $\gamma_Z(t_i)\in \partial D(M_i)$ which implies that $\gamma_Z(t_i)\in D(M')$ for some proper submodule $M'\subsetneq M_i$. The minimal such $M'$ must be Schurian which contradicts the uniqueness of $M_i$.

By Theorem \ref{thm: HN filtration}, we obtain an HN-stratification $\cW_Z(t_1),\cdots,\cW_Z(t_m)$ of $mod\text-\Lambda$. By Lemma \ref{lem: W=E(M)}, $\cW_Z(t_i)=\cE(M_i)$ proving that we have a finite HN-system.
\end{proof}


\subsection{Forward hom-orthogonality} A powerful reformulation of a finite HN-system.

\begin{defn}\label{def: forward hom-orthogonal}
We call a sequence of $\Lambda$-modules $M_1,\cdots,M_m$ \emph{weakly forward hom-orthogonal} if $\Hom_\Lambda(M_i,M_j)=0$ for all $i<j$. We call it \emph{maximal} if 
\begin{enumerate}
\item It cannot be embedded in a longer weakly forward hom-orthogonal sequence.
\item Each $M_i$ is Schurian.
\end{enumerate}
\end{defn}

An example of a weakly forward hom-orthogonal sequence of modules, with length $m=1$ is $M_1=\bigoplus S_i$, the sum of all simple modules. An example of a maximal forward hom-orthogonal sequence is given by taking $\Lambda$ of finite representation type so that its Auslander-Reiten quiver has no oriented cycles. Then the indecomposable $\Lambda$-modules are all Schurian and form a maximal forward hom-orthogonal sequence when they are ordered from right to left in the Auslander-Reiten quiver. In \cite{PartII} we use this observation to construct maximal green sequences of maximal length for many cluster-tilted algebras of finite type.

\begin{thm}\label{thm: HN statification is Schurian hom orthog}
Given a finite sequence of Schurian $\Lambda$-modules $M_1,\cdots,M_m$, the following are equivalent.
\begin{enumerate}
\item The sequence is maximal forward hom-orthogonal.
\item $\cE(M_1),\cdots,\cE(M_m)$ is a finite HN-system for $mod\text-\Lambda$.
\end{enumerate}
\end{thm}

\begin{proof}
 The statement that (2) implies that $M_i$ are forward hom-orthogonal is well-known and very easy \cite{R} Lemma 2.3. Indeed, if $\Hom(M_i,M_j)\neq0$ for $i<j$ then the HN-filtration of $M_i\oplus M_j$ would not be unique. The maximality of this forward hom-orthogonal sequence is also very easy. If there were an indecomposable module $X$ missing, the HN-filtration of $X$ would give nonzero morphisms $M_i\to X$ and $X\to M_j$ for $i<j$ forcing $X$ to be inserted before $M_i$ and after $M_j$ which is impossible. Thus $(2)\Rightarrow (1)$. The converse $(1)\Rightarrow (2)$ is also easy but we will go through it to make sure it is stated correctly.

Let $(M_i)$ be maximal forward hom-orthogonal and let $X$ be any module. Then, there is at least one $k$ so that $\Hom(M_k,X)\neq0$. Let $k$ be minimal and let $f:M_k\to X$ be nonzero. Then $f$ must be a monomorphism. Otherwise, by minimality of $k$, the image of $f$ would fit between $M_k$ and $M_{k-1}$ contradicting the maximality of $(M_i)$. So, we have a short exact sequence $M\hookrightarrow X\twoheadrightarrow Y$.

By induction on the length of $X$, we have an HN-filtration $Y_1\subseteq Y_2\subseteq \cdots\subseteq Y_m=Y$ of $Y=\coker f$ where $Y_i/Y_{i-1}\in \cE(M_i)$. Then it suffices to show:\vskip.2cm

\underline{Claim}: $Y_i=0$ for all $i<k$.\vskip.2cm

Then we obtain an HN-filtration: $0\subset X_k\subset X_{k+1}\subset\cdots\subset X_m=X$ of $X$ where each $X_j$ is the inverse image of $Y_j$ in $X$.

Pf: Let $j$ be minimal so that $Y_j\neq 0$ and suppose $j\le k$. Then we will show that $j=k$. Let $E\subseteq X$ be the inverse image of $Y_j$ in $X$. Then we have an exact sequence:
\[
	0\to M_k\to E\to Y_j\to 0
\]
\begin{enumerate}
\item $\Hom_\Lambda(M_i,E)=0$ for $i<k$ since $E\subseteq X$ and $\Hom(M_i,X)=0$.
\item $\Hom_\Lambda(E,M_p)=0$ for $p>k$ since $\Hom_\Lambda(M_k,M_p)=0$ making any map $E\to M_p$ factor through $Y_j\in\cE(M_j)$. But $j\le k<p$. So, the map is zero.
\end{enumerate}
By maximality of $(M_i)$ there must be a nonzero map $g:E\to M_k$. If the restriction of $g$ to $M_k$ is nonzero then $E=M_k\oplus Y_j$ and $\Hom_\Lambda(M_j,E)\neq0$ which implies $\Hom_\Lambda(M_j,X)\neq0$ since $E\subseteq X$. So, $j=k$. If $g|M_k=0$ then $\Hom_\Lambda(Y_j,M_k)\neq0$ showing again that $j=k$. So, $j=k$ in all cases. The proves the Claim. The Theorem follows.
\end{proof}

Demonet has shown \cite{KD} that maximal hom-orthogonal sequences of Schurian objects are in bijection with maximal chains of torsion classes. Together with Theorem \ref{thm: HN statification is Schurian hom orthog} this implies that HN-statifications as discussed in Theorem \ref{thm: HN statification is Schurian hom orthog} are in bijections with maximal chains of torsion classes. Li and Liu \cite{LL} have constructed a bijection between all three sets for any abelian length category.

%
%

\section{Finite Z and maximal green sequences}\label{sec4}

In this section we restrict to the case when $\Lambda$ is the finite dimensional hereditary algebra. The main result is that a finite green nonlinear stability function gives a maximal green sequence (MGS) for $\Lambda$ and all MGS's are given in this way. 

We first review the characterization of MGS's and more general reddening sequences by exceptional wall crossing sequences from \cite{IOTW2}. In particular, we already know that reddening sequences are given by paths which meet finitely many exceptional semistability sets $D(M_i)$ at distinct times $t_i$ where exceptional means that $M_i$ are exceptional modules. The only thing left to prove is that such a path cannot meet any $D(M)$ for nonexceptional $M$. The proof will be reduced to Theorem \ref{lemma A} which we prove in the last section.

We also show that, in the hereditary case, a MGS is equivalent to a maximal forward hom-orthogonal sequence. This implies Theorem \ref{main theorem} which says that all five notions of stability outlined in the introduction are equivalent in the hereditary case.


\subsection{Review of the cluster complex}\label{ss41} The cluster complex is well-known and there are several different constructions. (See, e.g., \cite{Reading}, \cite{Hubery}.) Here we follow \cite{IOTW2}.

Let $\Lambda$ be a hereditary algebra over $K$. Recall that an \emph{exceptional module} is a finite dimensional module $M$ which is \emph{rigid}, i.e., $\Ext_\Lambda^1(M,M)=0$ and Schurian. For example, the simple modules $S_i$ and their projective covers $P_i$ are exceptional. Exceptional modules are indecomposable and uniquely determined by their dimension vectors which are called (positive) \emph{real Schur roots} and, for each positive real Schur root $\beta$ we used the notation $M_\beta$ for the unique exceptional module with dimension vector $\beta$ and we use $D(\beta)$ to denote $D(M_\beta)$. We call these \emph{exceptional walls}.

Let $\cC_\Lambda$ be the cluster category \cite{BMRRT} of $\Lambda$ which is the orbit category of the bounded derived category of $mod$-$\Lambda$ by the equivalence $F=\tau^{-1}[1]$. We represent indecomposable objects of $\cC_\Lambda$ by representatives in the fundamental domain of $F$ which consists of $\Lambda$-modules and shifted projective modules $P_i[1]$. Then the exceptional objects of $\cC_\Lambda$ are the exceptional modules $M_\beta$ and these shifted indecomposable projective modules $P_i[1]$. To each exceptional object $T=M_\beta$ or $P_i[1]$ we associate the \emph{$g$-vector} $g(T)\in \mathbb Z^n$ which is characterized by the following properties where $f_i=\dim_K S_i$.
\begin{enumerate}
\item If $T=P_i[1]$ then $g(P_i[1])=-f_ie_i$ where $e_i$ is the $i$-th unit vector.
\item If $T=M_\beta$ and $\coprod n_iP_i\to\coprod m_iP_i\to M_\beta$ is a minimal projective presentation then the $i$-th coordinate of $g(M_\beta)$ is $f_i(m_i-n_i)$.
\end{enumerate}
Thus, dot product with $g(T)$ gives the \emph{Euler-Ringel pairing}:
\[
	g(T)\cdot \undim M=\left<\undim T,\undim M\right>=\dim_K\Hom_{\cD^b}(T,M)-\dim_K\Ext^1_{\cD^b}(T,M).
\]
Here we need to take Hom and Ext in the bounded derived category $\cD^b$ of $mod$-$\Lambda$ where, e.g., $\Ext^1_{\cD^b}(P[1],M)=\Hom_\Lambda(P,M)$. In this notation, we have the following.

\begin{thm}[Virtual Stability Theorem]\cite{IOTW2}\label{thm: virtual stability theorem}
$g(X)$ lies in the exceptional wall
\[
 D(M)=\{x\in \mathbb R^n\,|\, x\cdot\undim M=0,\,x\cdot\undim M'\le 0\ \forall M'\subset M\}
 \]
if and only if $\Hom_\Lambda(X,M)=0=\Ext^1_\Lambda(X,M)$.
\end{thm}

We repeat that, in \cite{IOTW2}, the sets $D(M)$ are defined only for exceptional $M$. In this paper $D(M)$ is defined for all modules $M$.

A \emph{cluster tilting object} for $\Lambda$ is an object $T=T_1\oplus \cdots \oplus T_n$ in the cluster category of $\Lambda$ with $n$ nonisomorphic components each of which is an exceptional object so that $\Ext^1_{\cC_\Lambda}(T,T)=0$. This condition of having no self-extensions in the cluster category $\cC_\Lambda$ is equivalent to the condition that it has no self-extensions in the derived category.

{
Following \cite{IOTW2} we define the \emph{$c$-vectors} of a cluster tilting object $T=T_1\oplus \cdots \oplus T_n$ to be the unique collection of vectors $c_i$ so that
\[
	c_i\cdot g_j=-\delta_{ij} \dim \End_{\cC_\Lambda}(T_i)
\]
This can also be written $C^tG=-D$ where $C,G$ are the matrices of $c$-vectors and $g$-vectors and $D$ is the diagonal matrix with entries $f_i=\dim_K\End_{\cC_\Lambda}(T_i)$.
The original equation in \cite{NZ} has the opposite sign: $C^tG=D$. It is shown in \cite{NZ}, \cite{IOTW2} that $C$ is an integer matrix.

We continue to follow \cite{IOTW2}. Using the $c$-matrix $C$ we can construct the exchange matrix of a cluster tilting object as follows. First, we define the \emph{initial exchange matrix} $B_0$ to be the $n\times n$ matrix with entries
\[
	b_{ij}=\dim_{F_i}\Ext^1_\Lambda(S_i,S_j)-\dim_{F_i}\Ext^1_\Lambda(S_j,S_i)
\]
where $F_i=\End_\Lambda(S_i)$. Then $DB_0$ is skew symmetric where $D$ is the diagonal matrix with entries $f_i=\dim_KF_i$. By Theorem 4.1.5(d) of \cite{IOTW2}, we can permute the indices so that this agrees with $f_i=\dim_K\End_{\cC_\Lambda}(T_i)$ above. Given any cluster tilting object $T=T_1\oplus\cdots\oplus T_n$ with components suitably ordered, the \emph{exchange matrix} of $T$ is
\[
	B=D^{-1}C^tDB_0C
\]
where $C$ is the $c$-matrix of $T$ defined above. Note that $DB$ is skew-symmetric. 

By \cite{BMRRT} and \cite{NZ}, this exchange matrix transforms according to the formulas of Fomin and Zelevinsky \cite{FZ}: There is, for any $k=1,\cdots,n$, a unique indecomposable object $T_k'\not\cong T_k$ so that the object $T'$ defined to be $T$ with summand $T_k$ replaced by $T_k'$ is a cluster tilting object. Let $G',C'$ be the matrices of $g$-vectors and $c$-vectors for $T'$ and $G,C$ for $T$ (keeping in mind that our $c$-vectors have the ``wrong sign''). The following theorem from \cite{BMRRT} is proved with these conventions in \cite{IOTW2}.

\begin{thm}\label{thm: exchange matrix transforms}
The entries $b_{ij}'$ of the exchange matrix $B'=D^{-1}C'^tDB_0C'$ for $T'$ are given by
\[
	b_{ij}=\begin{cases} -b_{ij} & \text{if } i=k\text{ or }j=k\\
   b_{ij}+\frac{b_{ik}|b_kj|+|b_{ik}|b_{kj}}2 & \text{otherwise}
    \end{cases}
\]
The entries $c_{ij}$ of the $c$-matrix $C'$ are given by
\[
	c_{ij}=\begin{cases} -c_{ij} & \text{if } j=k\\
   c_{ij}+\frac{c_{ik}|b_kj|-|c_{ik}|b_{kj}}2 & \text{otherwise}
    \end{cases}
\]
\end{thm}
}


For each cluster tilting object $T=\coprod T_i$, there is a conical simplex $R(T)\subset\mathbb R^n$ given as follows. The vertices of $R(T)$ are the $n$ rays generated by the $g$-vectors of the components of $T$. The faces of $R(T)$ are the subsets of the semistability sets $D(M_i)$ in the convex hull of these $n$ rays where $M_i=M_{\beta_i}$ are the unique exceptional modules having the property that, in the derived category, $\Hom_{\cD^b}(T_i,M_j)=0=\Ext_{\cD^b}^1(T_i,M_j)$ or, equivalently, $g(T_i)\in D(M_j)$ for all $i\neq j$.

The \emph{cluster complex} is defined to be the union of the conical simplices $R(T)$. This union is a \emph{simplicial fan} which means that the intersection of any two $R(T)\cap R(T')$ is either empty or a common face of codimension $\ge1$. In particular, any codimension one simplex of the cluster complex is the common face of exactly two simplices $R(T),R(T')$.


\begin{eg}
Take the modulated quiver of type $B_2$ given by
\[
	\mathbb R\xleftarrow{\mathbb C} \mathbb C
\]
with $F_1=\mathbb R$, $F_2=\mathbb C$ with $f_1=1,f_2=2$ and $M_{21}$ is the $\mathbb C$-$\mathbb R$-bimodule $M_{21}=\mathbb C$. Two cluster tilting objects for this quiver are $T=S_2\oplus P_1[1]$ and $T'=S_2\oplus I_1$ with $g$-vectors:
\begin{enumerate}
\item $g(P_1[1])=-f_1e_1=(-1,0)$.
\item $g(S_2)=(-2,2)$ since $P_1^2\to P_2\to S_2$ is a minimal presentation of $S_2$.
\item $g(I_1)=(-1,2)$ since $P_1\to P_2\to I_1$ is a minimal presentation of $I_1$.
\end{enumerate}
{For $T=S_2\oplus P_1[1]$ the $c$-vectors are $-\undim S_2=(0,-1)$, $\undim I_1=(1,1)$ since:
\[
	C^tG=\mat{0 & -1\\ 1&1} \mat{-2 & -1\\2 & 0}=\mat{-2 & 0 \\
	0 & -1}=-D.
\]
For $T'=S_2\oplus I_1$ the $c$-vectors are $\undim P_2=(2,1)$, $-\undim I_2=(-1,-1)$ by a similar calculation.} The cluster complex is given in Figure \ref{FigureB2}.
\begin{figure}[htbp]
\begin{center}
\begin{tikzpicture}[scale=1.3]
\coordinate (A) at (-2.8,0);
\coordinate (B) at (3,0);
\coordinate (B0) at (2.5,0);
\coordinate (C) at (0,-1.2);
\coordinate (D) at (0,2.2);
\coordinate (D0) at (0,-.9);
\coordinate (E) at (0,0);
\coordinate (F) at (-2.3,2.3);
\coordinate (F2) at (-1.15,2.3);
\coordinate (F3) at (0,2.2);
\coordinate (P1s) at (-1,0);
\coordinate (P1) at (1,0);
\coordinate (S2) at (-2,2);
\coordinate (S2L) at (-2,1.8);
\coordinate (I1L) at (-1.2,1);
\coordinate (I1) at (-1,2);
\coordinate (I11) at (-.9,2);
\coordinate (A1) at (-.53,1.2);
\coordinate (A2) at (0.2,1.5);
\coordinate (A3) at (0,1);
\coordinate (A4) at (.8,1.3);
\coordinate (AL) at (.8,1.3);
\draw[thick,<-] (A1)..controls (A2) and (A3)..(A4);
\draw (AL) node[right]{$D(P_2)$};
\draw[very thick] (A)--(B);
\draw[very thick] (C)--(D);
\draw[very thick] (E)--(F);
\draw[very thick] (E)--(F2);
\draw (B0) node[above]{$D(S_2)$};
\draw (D0) node[right]{$D(S_1)$};
\draw[fill] (P1s) circle[radius=2pt] node[below]{$g(P_1[1])$};
\draw[fill] (P1) circle[radius=2pt] node[above]{$g(P_1)$};
\draw[fill] (I1) circle[radius=2pt];
\draw[fill] (S2) circle[radius=2pt] (S2L) node[left]{$g(S_2)$} (I1L) node[left]{$D(I_1)$} (I11) node[right]{$g(I_1)$};
\end{tikzpicture}
\caption{The cluster complex for $B_2$ where $B_2:\mathbb R\leftarrow \mathbb C$. Since $S_1\subset I_1,P_2$, $D(I_1)$ and $D(P_2)$ occur only on the negative side of $D(S_1)$. Also, $g(I_1)\in D(P_2)$ since $\Hom(I_1,P_2)=0=\Ext^1(I_1,P_2)$. And
\[
	\left<I_1,P_2\right>=g(I_1)\cdot \undim P_2=(-1,2)\cdot (2,1)=0.\qquad\qquad
\]Similarly, $g(S_2)\in D(I_1)$ and $g(P_1[1]),g(P_1)\in D(S_2)$.
}
\label{FigureB2}
\end{center}
\end{figure}
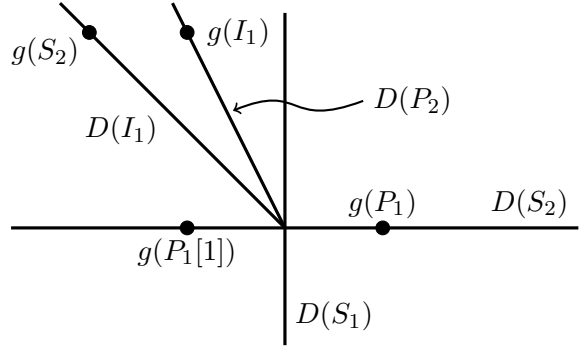
\end{eg}


\begin{thm}\cite{IOTW2}\label{thm: crossing the walls}
For every cluster tilting object $T=\bigoplus T_i$ of the cluster category of $\Lambda$ the codimension 1 wall of $R(T)$ are $D(\beta_i)$ for $\varepsilon_1\beta_1,\cdots,\varepsilon_n\beta_n$ the $c$-vectors of the cluster tilting object $T$ where $\varepsilon_i=\pm1$ is the negative of the sign of $g(T_i)\cdot \beta_i$.
\end{thm}

{
\begin{rem}
Theorems \ref{thm: crossing the walls} and \ref{thm: exchange matrix transforms} imply that the Fomin-Zelevinsky $c$-vectors $\varepsilon_i \beta_i$ of a mutation sequence give, up to sign, the walls $D(\beta_i)$ being crossed by the corresponding regions $R(T)$ in the cluster complex.
\end{rem}

Recall from \cite{IOTW2} that a \emph{reddening sequence} consists of a sequence of cluster tilting objects $T_0,\cdots,T_m$ where $T_0=\Lambda[1]$ and $T_m=\Lambda$ so that each $R(T_k)$ shares a wall, say $D(\beta_k)$ with $R(T_{k+1})$. A \emph{maximal green sequence} is a reddening sequence so that each mutation} $T_k\to T_{k+1}$ is \emph{green} in the sense that $R(T_{k+1})$ is on the positive side of $D(\beta_k)$. This gives a path $\gamma$ in $\mathbb R^n$ going through the regions $R(T_k)$ in order and passing through the walls $D(\beta_k)$. However, in this paper we have additional walls $D(M)$ for modules $M$ which may not be exceptional. Theorem \ref{lemma A} and Lemma \ref{lem: finite paths stay in the cluster complex} allow us to ignore these other walls.

\begin{customthm}{A}[Theorem \ref{cor: lemma a}]\label{lemma A}
If $M$ is not an exceptional module then $D(M)$ does not meet the interior of any codimension $0$ or $1$ face of the cluster complex.
\end{customthm}

This implies that a path which meets $D(M)$ for non-exceptional $M$ must either pass through infinitely many walls $D(\beta)$ of the cluster complex or it must pass through a simplex of codimension $\ge2$.

\begin{customlem}{B}\label{lem: finite paths stay in the cluster complex}
Let $\gamma$ be any path in $\mathbb R^n$ which meets only a finite number of exceptional walls $D(\beta)$, and at least one, and which is transverse to the cluster complex, i.e., is disjoint from the codimension $\ge2$ simplices. Then $\gamma$ does not meet $D(M)$ for any $M$ which is not exceptional.
\end{customlem}

\begin{proof}
The assumption that $\gamma$ meets at least one wall means that part of $\gamma$ lies in the union of the two $n-1$ simplices (with codimension 0) which contain that wall. The boundary of each $n-1$ simplex is a union of exactly $n$ faces which are codimension 1 simplices. On the other side of each face, there is another $n-1$ simplex. If $\gamma$ crosses $k$ walls then, by induction on $k$, it will be in the union of $k+1$ open codimension 0 simplices and $k$ open codimension 1 simplices. By Theorem \ref{lemma A}, $\gamma$ cannot meet any $D(M)$ for $M$ not exceptional.
\end{proof}


\subsection{Maximal green sequences} For $\Lambda$ hereditary, the finite, green nonlinear stability functions $Z_\bullet$ correspond to maximal green sequences:

\begin{thm}\label{Rem: nonlinear stability functions and reddening sequences}
Let $Z_t:K_0\Lambda\to\mathbb C$ be a finite, generic nonlinear stability function with stable pairs $(M_i,t_i)$. Then $\varepsilon_1\beta_1,\cdots,\varepsilon_m\beta_m$, where $\beta_i=\undim M_i$ and $\varepsilon=\pm1$, are the $c$-vectors of a reddening sequence of $\Lambda$ and all reddening sequence for $\Lambda$ are obtained in this way. Furthermore $\varepsilon_i$ is negative in the reddening sequence iff $(M_i,t_i)$ is red for $Z_\bullet$. 
\end{thm}

\begin{proof}
One direction follows easily from the lemmas. Given $Z_\bullet$ finite, by Lemma \ref{lem: finite paths stay in the cluster complex}, all $Z_\bullet$-stable modules are exceptional, say $M_k=M_{\beta_k}$. Lemma \ref{lem3} implies that the path $\gamma_Z$ starts in the ``green region'' where all coordinates are negative and ends in the ``red region'' where all coordinates are positive. Lemma \ref{lem1} implies that $\gamma_Z(t)$ is disjoint from any $D(M)$ except at time $t_k$ when $\gamma_Z(t_k)\in D({\beta_k})$. Thus, for each $k$, there is a cluster tilting object $T_k$ so that $\gamma(t)\in R(T_k)$ for $t_{k-1}<t<t_k$. Lemma \ref{lem2} implies that, at $t=t_k$, $\gamma_Z(t)$ crosses $D(\beta_k)$ in the green or red direction depending on whether the expression \eqref{eq: when (b,t) is Z-green} is $<1$ or $>1$, respectively. Therefore, $T_0,\cdots,T_m$ is a reddening sequence with $c$-vectors $\varepsilon_1\beta_1,\cdots,\varepsilon_m\beta_m$.

Conversely, let $T_0,\cdots,T_m$ be a reddening sequence, i.e., a sequence of mutation from $T_0=\Lambda[1]$ to $T_m=\Lambda$. Suppose that $\gamma:\mathbb R\to\mathbb R^n$ is a smooth path starting in $R(T_0)$, ending in $R(T_m)$ and passing through the regions $R(T_k)$ in order of $k$ and transverse to the exceptional walls separating these regions. We may assume $\gamma(t)=t(f_1,\cdots,f_n)$ for $|t|$ large where $f_i=\dim_KS_i$ since these points lie in $R(\Lambda[1])$ and $R(\Lambda)$. Equivalently, the path $a_t=t(f_1,\cdots,f_n)-\gamma(t)$ in $\mathbb R^n$ has bounded support. So
\[
	\gamma(t)= -a_t + t(f_1,\cdots,f_n).
\]
This corresponds to the nonlinear stability function $Z_t(x)=a_t\cdot x+b_t\cdot x\sqrt{-1}$ where $b_t=(f_1,\cdots,f_n)$ for all $t\in\mathbb R$. This clearly satisfies the conditions of Definition \ref{def: nonlinear stability function}.

By construction, the path $\gamma$ crosses exactly $m$ exceptional walls $D(M_1),\cdots,D(M_m)$ at times $t_1<\cdots<t_m$ with $M_i=M_{\beta_i}$ exceptional. By Lemma \ref{lem: finite paths stay in the cluster complex} $\gamma$ crosses no other walls. Therefore, $\gamma$ is a finite reddening path corresponding to a finite stability function $Z_\bullet$. Lemma \ref{lem1} implies that $M_i$ are $Z_{t_i}$-stable and there are no other $Z_\bullet$ semistable modules. Lemma \ref{lem2} implies that $(M_{\beta_i},t_i)$ is green or red for $Z_\bullet$ iff $\gamma$ passes transversely through $D(\beta_i)$ in the green or red direction respectively at time $t_i$. Thus the sequence of stable roots $\beta_i$ given by $Z_\bullet$ is the same up to the correct sign as the $c$-vectors of the arbitrary reddening sequence that we started with.
\end{proof}

\begin{cor}\label{thm: nonlinear Z correspond to MGSs}
Let $Z_t:K_0\Lambda\to\mathbb C$ be a finite, green nonlinear stability function with stable pairs $(M_i,t_i)$. Then each $M_i=M_{\beta_i}$ is exceptional and the sequence of real Schur roots $\beta_1,\beta_2,\cdots,\beta_m$ in increasing order of $t_i$, form the $c$-vectors of a unique maximal green sequence for $\Lambda$. Furthermore, all maximal green sequences for $\Lambda$ are given in this way.\qed
\end{cor}


\subsection{Hom-orthogonality and MGSs} In this subsection we show that a maximal forward hom-orthogonal sequence is equivalent to a maximal green sequence. This is a reformulation of the well-known statement that a maximal green sequence is equivalent to a maximal chain in the poset of finitely generated torsion classes in $mod\text-\Lambda$. An alternate proof which works for quivers with potential is explained in \cite{PartII}.

{\begin{thm}\label{thm: hom-orthog = mgs}
Let $M_1,\cdots,M_m$ be a finite sequence of Schurian modules over a hereditary algebra $\Lambda$. Then the following are equivalent.
\begin{enumerate}
\item The sequence is maximal forward hom-orthogonal.
\item $\beta_k=\undim M_k$ are the $c$-vectors of a maximal green sequence for $\Lambda$. In particular, all $M_k$ are exceptional.
\end{enumerate}
\end{thm}}

The proof of this theorem occupies the rest of this section.

\begin{proof}
We have already shown that $(2)\Rightarrow (1)$ since, by Corollary \ref{thm: nonlinear Z correspond to MGSs}, any MGS is given by a finite green path which, by Theorem \ref{thm: finite green (a) implies finite (c)} gives an HN-system which is equivalent to (1) by Theorem \ref{thm: HN statification is Schurian hom orthog}. So, it suffices to show that $(1)\Rightarrow (2)$. We use the representation theoretic definition of a maximal green sequence explained in subsection \ref{ss41}. We use Lemma \ref{lem: finite paths stay in the cluster complex} to insure that all walls that we encounter are exceptional walls.

We recall some standard cluster theory in the language of \cite{InTh}, \cite{IOTW2} and \cite{BMRRT}. In checking that the proofs in \cite{InTh} work for modulated quivers  we noticed one misprint: In the proof of Lemma 2.8 in \cite{InTh}, the reference should be to ``proof of Lemma VI.6.1 in \cite{ASS} using the trace of $R$ in $U_0$, not $g(R)$.''

Recall that a \emph{support tilting module} is a $\Lambda$-module $T$ so that $\Ext_\Lambda(T,T)=0$ and the number of (nonisomorphic) summands of $T$ is equal to the size of its support. In the cluster category \cite{BMRRT}, $T$ can be completed to a cluster tilting object with $n$ elements, $n$ being the number of vertices of the modulated quiver of $\Lambda$, by adding the shift $P_i[1]$ of the projective cover $P_i$ of each vertex $i$ not in the support of $T$. Let $R(T)$ denote the simplicial cone of the cluster tilting object $T\oplus \coprod P_i[1]$.

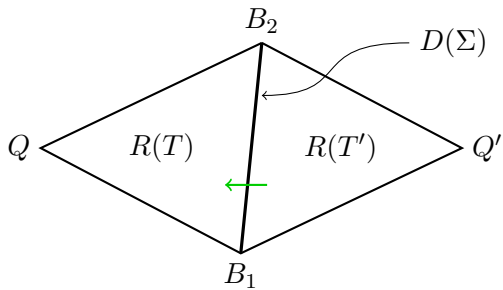
\begin{figure}[htbp]
\begin{center}
\begin{tikzpicture}[scale=.7]
\coordinate (A0) at (0,2);
\coordinate (A1) at (8,2);
\coordinate (B1) at (3.8,0);
\coordinate (B2) at (4.2,4);
\coordinate (B0) at (4.2,3);
\coordinate (L) at (7,4);
\coordinate (C1) at (5,4);
\coordinate (C2) at (6,3);
\coordinate (L1) at (2.3,2);
\coordinate (L2) at (5.7,2);
\coordinate (G1) at (4.3,1.3);
\coordinate (G2) at (3.5,1.3);
\draw[thick] (B2)--(A1)--(B1)--(A0)--(B2);
\draw[very thick](B1)--(B2);
\draw (L) node[right]{$D(\Sigma)$};
\draw[->] (L).. controls (C1) and (C2)..(B0);
\draw(L1)node{$R(T)$};
\draw(L2)node{$R(T')$};
\draw(A0)node[left]{$Q$};
\draw(A1)node[right]{$Q'$};
\draw(B1)node[below]{$B_1$};
\draw(B2)node[above]{$B_2$};
\draw[thick, color=green!80!black,->](G1)--(G2);
\end{tikzpicture}
\caption{$T'\mapsto T$ is a green mutation if and only if the wall separating $R(T)$ and $R(T')$ is $D(\Sigma)$ for some simple object $\Sigma$ in the wide subcategory corresponding to $Gen(T)$ and $Q$ is the unique split projective in $Gen(T)$ which maps onto $\Sigma$.}
\label{Fig: mutation as wall crossing}
\end{center}
\end{figure}

Recall from \cite{InTh} that, for each support tilting module $T$, there is a corresponding torsion class $Gen(T)$ in which $T$ is the sum of the $\Ext$-projective objects. There are two kinds of components of $T$: \emph{split projectives} and \emph{nonsplit projectives} where, by \cite{InTh}, Proposition 2.16, the split projectives of $T$ are the projective objects in the wide subcategory $\cW(T)$ of $mod\text-\Lambda$ corresponding to $Gen(T)$ and, by \cite{InTh}, Proposition 2.24, the nonsplit projectives are those components of $T$ which do not map nontrivially to any object in $\cW(T)$. We will need the following which follows from \cite{InTh}, Theorem 2.35 and Proposition 2.24.

\begin{lem}\label{lem: mutation through Sigma}
The ``red'' walls $D(\Sigma_k)$ of the region $R(T)$ are those opposite (the $g$-vector of) the split projective summands $Q_k$ of $T$ so that $\Sigma_k$ is the top of the projective object $Q_k$ in the wide subcategory $\cW(T)$ of $Gen(T)$.
\end{lem}

\begin{proof}
Let $\Sigma_k\in\cW(T)$ be the simple top of the projective object $Q_k$. Then 
\[
\Hom_\Lambda(T/Q_k,\Sigma_k)=0=\Ext_\Lambda^1(T/Q_k,\Sigma_k)
\]
and the support of $\Sigma_k$ is contained in the support of $T$. Therefore, $D(\Sigma_k)$ is the unique wall that contains all vertices of the region $R(T)$ not equal to $Q_k$. Since $\Hom_\Lambda(Q_k,\Sigma_k)\neq0$, the wall is ``red'', i.e., the corresponding $c$-vector is negative.

If $P$ is a nonsplit projective summand of $T$ or $P=P_i[1]$, the wall opposite $g(P)$ in $R(T)$ is $D(M)$ for some $M$. However, $D(M)$ contains all of the split projectives in $Gen(T)$. Therefore, no object of $Gen(T)$ can map nontrivially to $M$. In particular, $P$ cannot map nontrivially to $M$. So, $D(M)$ is a ``green'' wall of $R(T)$.
\end{proof}

Recall that a support tilting module $T$ is a \emph{green mutation} of $T'$ if either $T=T'\oplus A$ for some exceptional module $A$ or $T=Q'\oplus B$, $T'=Q'\oplus B$ where $Q,Q'$ are exceptional modules so that $\Ext_\Lambda^1(Q,Q')=0$. We say that $T$ is a mutation of $T'$ \emph{through} $\Sigma$ if $\Sigma$ is the unique exceptional module with support in the support of $T$ so that $\Hom_\Lambda(B,\Sigma)=0=\Ext_\Lambda^1(B,\Sigma)$ where $B$ is the largest common summand of $T,T'$. In that case, the $c$-vector of the mutation is $\undim \Sigma$ and $D(\Sigma)$ is the wall that separates the regions $R(T)$ and $R(T')$ in the cluster fan as shown in Figure \ref{Fig: mutation as wall crossing}. (See \cite{IOTW2}.)

We recall \cite{InTh} that $Gen(\cW(T))=Gen(T)$.

\begin{lem}\label{a lemma}
Let $M_1,\cdots,M_m$ be a maximal forward hom-orthogonal sequence of Schurian modules. Then,
for every $0\le k\le m$ there is a support tilting module $T_k$ satisfying the following.
\begin{enumerate}
\item[(a)] $M_1,\cdots,M_k\in Gen(T_k)$.
\item[(b)] $\Hom_\Lambda(T_k,M_\ell)=0$ for all $\ell>k$. (So, $T_0=0$.)
\item[(c)] $M_{k+1},\cdots,M_m$ form a sequence of green mutations from $T_k$ to $\Lambda$. (So, $T_m=\Lambda$.)
\item[(d)] $M_k$ is a simple object of the wide subcategory $\cW(T_k)$.
\item[(e)] $T_{k-1}\mapsto T_k$ is a green mutation through $D(M_k)$. (So, $M_k$ is exceptional.)
\end{enumerate}
\end{lem}

For $k=0$ we see that $(c)$ implies that $\undim M_1,\cdots,\undim M_m$ are the $c$-vectors of a maximal green sequence which completes the proof of Theorem \ref{thm: hom-orthog = mgs}. Therefore, it suffices to prove this lemma.

\begin{proof}
By downward induction on $k$. When $k=m$ we have $T_m=\Lambda$ and $(a),(b),(c)$ hold trivially. So, suppose that $k\le m$ and that $(a)_k,(b)_k,(c)_k$ hold as well as $(c)_{k+1},(d)_{k+1}$ (which we take to be vacuous when $k=m$). It follows from $(a)_k,(b)_k$ and the maximality of $(M_i)$ that $M_1,\cdots,M_k$ is a maximal forward hom-orthogonal sequence of exceptional modules in $Gen(T_k)$.

$(d)_k$: Let $\Sigma_i$ be the simple objects of the wide subcategory $\cW(T_k)$. For each $i$ there is a $j_i\le k$ so that $\Hom(M_{j_i},\Sigma_i)\neq 0$, otherwise, we could insert $\Sigma_i$ after $M_k$ in the sequence $M_1,\cdots,M_k$ contradicting its maximality. Let $f_i:M_{j_i}\to \Sigma_i$ be nonzero. Since $M_{j_i}\in Gen(T_k)=Gen(\cW(T_k))$, the image of $f_i$ is in $\cW(T_k)$. Since $\Sigma_i$ is simple, $f_i$ must be onto. But $M_k\in Gen(\cW(T_k))$. So, there is a nonzero map $\Sigma_i\to M_k$ for some $i$. The composition $M_{j_i}\twoheadrightarrow \Sigma_i\to M_k$ is nonzero. So, we must have $M_{j_i}=M_k$. Since $M_k$ is Schurian, the composition $M_k\twoheadrightarrow \Sigma_i\to M_k$ must be an isomorphism. So, $\Sigma_i\cong M_k$, proving $(d)_k$.

$(e)_k$: By Lemma \ref{lem: mutation through Sigma}, we can let $T_{k-1}$ be the support tilting module obtained by mutation of $T_k$ though $M_k$.

$(c)_{k-1}$ holds by $(c)_k$ and the construction of $T_{k-1}$.

$(b)_{k-1}$: Let $B$ be the common summand of $T_k,T_{k-1}$ so that $T_k=B\oplus Q$ and $T_{k-1}=B\oplus Q'$. Since $B$ is a summand of $T_{k}$, $\Hom_\Lambda(B,M_\ell)=0$ for all $\ell>k$. Since the $g$-vectors of the components of $B$ lie on $D(M_k)$ we also have $\Hom_\Lambda(B,M_k)=0$. By Lemma \ref{lem: mutation through Sigma}, $Q'$ is a nonsplit projective in $Gen(T_{k-1})$ and $B$ is a generator of $Gen(T_{k-1})$. Therefore, $\Hom_\Lambda(T_{k-1},M_\ell)=0$ for all $\ell\ge k$.

$(a)_{k-1}$: Since $Q'$ is a nonsplit projective in $Gen(T_{k-1})$, $Gen(T_{k-1})=Gen(B)\subseteq Gen(T_k)$. To prove that $M_1,\cdots,M_{k-1}$ lie in $Gen(B)$, suppose not. Let $j$ be minimal so that $M_j$ is not in $Gen(B)$. Let $X\subset M_j$ be the trace of $B$ in $M_j$, i.e., the largest submodule of $M_j$ which lies in $Gen(B)$. Since this is a torsion class, $\Hom_\Lambda(B,M_j/X)=0$. But $M_j/X\in Gen(T_k)$ and $B$ is Ext-projective in this category. Therefore, $\Ext_\Lambda^1(B,M_j/X)=0$. This implies that the $g$-vector of $B$ lies in $D(M_k)$. So, $M_j/X\in add(M_k)$. So, $\Hom_\Lambda(M_j,M_k)=\neq0$ which is not possible for $j<k$. This proves $(a)_{k-1}$.

By induction on $m-k$, the proof of the Lemma is complete. 
\end{proof}

Lemma \ref{a lemma}(c) for $k=0$ shows that $(1)\Rightarrow (2)$ completing the proof of Theorem \ref{thm: hom-orthog = mgs}.
\end{proof}

%
%
\setcounter{section}4

\section{Proof of Theorem A}\label{ss: Lemma B proof}\label{sec5}

We prove Theorem \ref{lemma A} (Theorem \ref{cor: lemma a} below) as a consequence of Lemma \ref{prop: Lemma b}. Here $\Lambda$ is a finite dimensional hereditary algebra over a field $K$.

\begin{lem}\label{lem: hom orthogonal}
Let $X,Y$ be $\Lambda$-modules so that the interiors of $D(X),D(Y)$ meet as some point $x_0$.
Then $X,Y$ are hom orthogonal.
\end{lem}

\begin{proof}
By assumption the linear stability function whose corresponding path goes through $x_0$ has perturbations which go through $D(X),D(Y)$ in either order. So, $X,Y$ are hom-orthogonal.
\end{proof}

\begin{lem}\label{b-lemma 0}
Let $\sigma$ be a codimension $k$ simplex in the cluster complex of $\Lambda$. Then $\cW(\sigma)$ is equal to $\cW_0(\sigma)$, the wide subcategory of $mod$-$\Lambda$ generated by the exceptional objects of $\cW(\sigma)$.
\end{lem}

\begin{proof} Since $codim\,\sigma=k$, $\cW_0(\sigma)$ has $k$ simple objects $S_1,\cdots,S_k$. Let $M\in \cW(\sigma)$, $M\notin\cW_0(\sigma)$. Consider first the case when $M$ is a subobject of some $X\in \cW_0(\sigma)$.

The condition $\sigma\subseteq D(M)$ implies that $\undim M=\sum \lambda_i S_i$ for some $\lambda\in\mathbb Z$. Since $\undim M$ has nonnegative entries, at least one $\lambda_i$ is positive. Then $\left< \undim P_i,\undim M\right>=\lambda_i\left< \undim P_i,\undim S_i\right> >0$ where $P_i$ is the projective cover of $S_i$ in $\cW_0(\sigma)$. So, there is a nonzero morphism $f:P_i\to M$ with image $Y\subset M\subset X$. Since $P_i,X$ both lie in $\cW_0(\sigma)$, so do $Y$ and $X/Y$. Then the exact sequence $Y\to M\to M/Y\subset X/Y$ shows, by induction on the size of $M$ that $M\in\cW_0(\sigma)$.

Now consider a general element $M\in\cW(\sigma)$. As before there is a nonzero map $f:P_i\to M$. By the previous case, $\ker f\in\cW_0(\sigma)$. So, the image of $f$ also lies in $\cW_0(\sigma)$. But $\coker f$ is an object of $\cW(\sigma)$ which is smaller than $M$. So, it also lies in $\cW_0(\sigma)$. Since $\cW_0(\sigma)$ is closed under extension, $M$ lies in $\cW_0(\sigma)$ as claimed. 
\end{proof}

\begin{lem}\label{prop: Lemma b}
If $D(M)$ meets the interior of a cluster simplex $\sigma$ then $D(M)$ contains $\sigma$. Equivalently, $M\in \cW(\sigma)$.
\end{lem}

\begin{proof} Suppose not. Let $M,\sigma$ form a counterexample so that $(\dim M,\dim \sigma)$ is minimal in lexicographic order. Let $x_0\in D(M)\cap int\,\sigma$. Equivalently, $M\in \cW(x_0)$. \vskip.2cm

\underline{Claim 1}: $x_0$ lies in the interior of $D(M)$. Equivalently, $M$ is a simple object of $\cW(x_0)$.\vskip.2cm

Proof: If not, $x_0\in D(M')$ for some $0\neq M'\subseteq M$. Then $M',M/M'$ lie in $\cW(x_0)$ and thus also in $\cW(\sigma)$ by induction on $\dim M$. Since $\cW(\sigma)$ is closed under extension, $M\in \cW(\sigma)$ and $M$ is not a counterexample. \vskip.2cm

Claim 1 implies that $F_M=\End(M)$ is a division algebra. So, $M$ is indecomposable.\vskip.2cm

\underline{Claim 2}: $\dim\sigma=1$. So, $codim\, \sigma=n-2$.\vskip.2cm

Proof: Since any $x_0\in D(M)\cap int\,\sigma$ lies in the interior of $D(M)$, $D(M)\cap \sigma=H(M)\cap \sigma$. The hyperplane $H(M)$ cuts $\sigma$ into two parts. Taking two vertices of $\sigma$ on opposite sides of $H(M)$ we get an edge of $\sigma$ whose interior meets $D(M)$. By minimality of $\dim\sigma$ this edge is equal to $\sigma$.
\vskip.2cm

Since $\sigma$ is an edge in the cluster complex of $\Lambda$, there are two ext-orthogonal exceptional objects $T_1,T_2$ of the cluster category of $\Lambda$ so that their $g$-vectors $E^t\undim T_i$ form the endpoints of $\sigma$. Let $T=T_1\oplus T_2$. Then $T^\perp$, the full subcategory of $mod$-$\Lambda$ of all $X$ so that $\Hom_\Lambda(T,X)=0=\Ext^1_\Lambda(T,X)$ is the wide subcategory $\cW_0(\sigma)$ of $mod$-$\Lambda$ of rank $n-2$ spanned by $n-2$ simple objects $S_1,\cdots,S_{n-2}$ and generated by the corresponding projective objects $P_1,\cdots,P_{n-2}$. By Lemma \ref{b-lemma 0}, $\cW_0(\sigma)=\cW(\sigma)$.\vskip.2cm

\underline{Claim 3}: There is an extension $P\hookrightarrow E\twoheadrightarrow M$ in $\cW(x_0)$ with the following properties.
\begin{enumerate}
\item $P$ is a projective object of $\cW(\sigma)=\cW_0(\sigma)$.
\item $E^\perp$ contains $\cW(\sigma)$. Equivalently, $E$ is in the wide subcategory spanned by $T$.
\item $E$ is indecomposable.
\end{enumerate}\vskip.2cm

Suppose for a moment that Claim 3 holds. Let $\cW_0$ be the wide subcategory of $mod$-$\Lambda$ spanned by $T$. If $D(E)\cap \sigma$ contains only $x_0$ then, up to reordering, we must have $\Hom_\cD(T_1,E)\neq 0$ and $\Ext^1_\cD(T_2,E)\neq0$ where $\cD$ is the bounded derived category of $\cW_0$. But this is not possible since $T_1,T_2$ are consecutive objects in the Auslander-Reiten sequence of the cluster category of the rank 2 hereditary abelian category $\cW_0$. Thus Claim 3 leads to a contradiction proving the proposition.\vskip.2cm

Proof of Claim 3: The extension $P\to E\to M$ is a lifting of the universal extension of $M$ by the simple objects of $\cW(\sigma)$. More precisely, for each $i$, let $e_{ij}\in\Ext^1_\Lambda(M,S_i)$, $j=1,\cdots,m_i$ form a basis of $\Ext^1_\Lambda(M,S_i)$ over the division algebra $F_M=\End_\Lambda(M)$. Since $\Ext^1$ is right exact, $e_{ij}$ lift to elements $\tilde e_{ij}\in\Ext^1_\Lambda(M,P_i)$. Let $P=\coprod P_i^{m_i}$ and let $P\to E\to M$ be the extension of $M$ by $P$ corresponding to the element of $\Ext^1_\Lambda(M,P)=\coprod\Ext^1_\Lambda(M,P_i)^{m_i}$ with $ij$ term $\tilde e_{ij}$. Then for any simple object $S_i$ of $\cW(\sigma)$ the connecting homomorphism in the six term sequence:
\[
	0\to (M,S_i)\to (E,S_i)\to (P,S_i)\xrightarrow\cong \Ext^1_\Lambda(M,S_i)\to \Ext^1_\Lambda(E,S_i)\to \Ext^1_\Lambda(P,S_i)\to 0
\]
is an isomorphism where $(X,Y)$ is short for $\Hom_\Lambda(X,Y)$. Since $\Hom_\Lambda(M,S_i)=0$ by Claim 1 and Lemma \ref{lem: hom orthogonal} and $\Ext^1_\Lambda(P,S_i)=0$ since $P$ is projective, we conclude that $E^\perp$ contains each $S_i$ and therefore all of $\cW(\sigma)$ proving (2).

Since $M$ is hom-orthogonal to all $S_i$, $\Hom_\Lambda(M,P)=0$. So, $\Hom_\Lambda(E,M)\cong \Hom_\Lambda(M,M)$ is one-dimensional over $F_M$. Therefore, $E$ has one component which maps onto $M$ and any other component of $E$ must lie in $P$. But $\Hom_\Lambda(E,P)=0$ by (2). So, $E$ has only one component. This proves (3) in Claim 3 and completes the proof of the proposition.
\end{proof}

\begin{thm}[Theorem \ref{lemma A}]\label{cor: lemma a}
Given a hereditary algebra $\Lambda$ and a nonrigid module $M$, the semistability set $D(M)$ does not meet the interior of any simplex of the cluster complex of codimension $0$ or $1$.
\end{thm}

\begin{proof}
Since $D(A\oplus B)=D(A)\cap D(B)$ (Lemma \ref{lem: W(S) is closed under summands}), we may assume that $M$ is indecomposable. If $D(M)$ meets the interior of a simplex $\sigma$ then, by Lemma \ref{prop: Lemma b}, $D(M)$ contains $\sigma$. Equivalently, $M\in\cW(\sigma)$. This is not possible if $\sigma$ has full dimension since $D(M)$ lies in a hyperplane. So, $codim\, \sigma=1$ and $\sigma\subseteq D(M_\beta)$ for some exceptional module $M_\beta$. By Lemma \ref{b-lemma 0}, $\cW(\sigma)=\cW_0(\sigma)=add\,M_\beta$. So, $M\cong M_\beta$ is rigid, contrary to assumption.
\end{proof}

\section*{Acknowledgements}

This paper is an expanded version of a lecture given in the workshop on ``Quivers, cluster algebras and mathematical physics'' at the Chinese University in Hong Kong in December, 2016. This was an extremely productive meeting and I am grateful to Man Wai Cheung and all the organizers for hosting this wonderful event. In particular, I want to thank Yang-Hui He, Gregg Musiker and Yu Qiu for some very insightful questions and comments which influenced this work.

I also benefited from conversations with Hipolito Treffinger and Osamu Iyama after the first draft of this paper was written. I should also thank Tom Bridgeland for a conversation many years ago when we compared our respective definitions of wall crossing. The anonymous referee also deserves praise for greatly improving the presentation of the material.

Finally, I would like to thank Lutz Hille for persuading me to look at the linearity problem in 2013 and reminding me more recently of the difference between a central charge and a classical slope function.




\begin{thebibliography}{aa}

\bibitem{AI} PJ Apruzzese and Kiyoshi Igusa, \emph{Stability conditions for affine type A}, arXiv:1804.09100. 


\bibitem{ASS} I. Assem, D. Simson and A. Skowro\'nski, Elements of the representation theory of associative algebras, 1: Techniques of representation theory, London Mathematical Society Student Texts, vol. 65 (Cambridge University Press, Cambridge, 2006). 



\bibitem{B}Tom Bridgeland, \emph{Stability conditions on triangulated categories}, Ann. Math. \textbf{166}, No. 2 (Sep., 2007), pp. 317--345. 

\bibitem{B2} Tom Bridgeland, \emph{Spaces of stability conditions}, Algebraic geometry--Seattle 2005. Part 1 (2009): 1--21. 


\bibitem{BDP}
Thomas Br\"{u}stle, Gr\'egoire Dupont, and  Matthieu P\'erotin, \emph{On maximal green sequences}, Int Math Res Notices (2014), 4547--4586. 

\bibitem{BHIT}
Thomas Br\"ustle, Stephen Hermes, Kiyoshi Igusa and Gordana Todorov, \emph{Semi-invariant pictures and two conjectures on maximal green sequences}, J Algebra {\bf 473}, March 2017, 80--109. 

\bibitem{BST}Thomas Br\"ustle, David Smith, Hipolito Treffinger, \emph{Stability conditions, tau-tilting theory and maximal green sequences}, arXiv:1705.08227 (2017). 

\bibitem{BST2} Thomas Br\"ustle, David Smith, Hipolito Treffinger, \emph{Stability conditions and maximal green sequences in abelian categories}, arXiv:1805.04382 (2018). 

\bibitem{BMRRT}
Aslak~Bakke Buan, Robert~J. Marsh, Markus Reineke, Idun Reiten, and Gordana Todorov, \emph{Tilting theory and cluster combinatorics}, Adv. Math. \textbf{204} (2006), no.~2, 572--618. 

\bibitem{BR} M. C. R. Butler and C. M. Ringel, \emph{Auslander-Reiten sequences with few middle terms and applications to string algebras}, Comm. Algebra 15 (1987), 145--179. 

  
\bibitem{DW}
Harm Derksen and Jerzy Weyman, \emph{Semi-invariants of quivers and saturation for {L}ittlewood-{R}ichardson coefficients}, J. Amer. Math. Soc. \textbf{13} (2000), no.~3, 467--479 (electronic). 

\bibitem{DW2}
Harm Derksen and Jerzy Weyman, \emph{On the canonical decomposition of quiver representations}, Compositio Mathematica, September 2002, Volume 133, Issue 3, 245--265. 


\bibitem{FZ}
Sergey Fomin and Andrei Zelevinsky, \emph{Cluster algebras. {IV}. {C}oefficients}, Compos. Math. \textbf{143} (2007), no.~1, 112--164. 

\bibitem{FR} H. Franzen and M. Reineke, \emph{Semi-stable Chow-Hall algebras of quivers and quantized Donaldson-Thomas invariants}, arXiv:1512.03748. 

\bibitem{GSZ} T.L. G\'omez, I. Sols and A. Zamora, \emph{A GIT interpretation of the Harder-Narasimhan filtration}, A. Rev Mat Complut (2015) 28: 169--190. 

\bibitem{HN} G. Harder and M. S. Narasimhan, \emph{On the cohomology groups of moduli spaces of vector bundles on curves}, Math. Ann 212 (1974/75), 215--248. 

\bibitem{Hubery} Andrew Hubery, \emph{The cluster complex of an hereditary artin algebra}, Algebras and Representation Theory, December 2011, Volume 14, Issue 6, 1163--1185. 

\bibitem{IOTW} Kiyoshi Igusa, Kent Orr, Gordana Todorov and Jerzy Weyman, \emph{Cluster complexes via semi-invariants}, Compos. Math. 145 (2009), no. 4, 1001--1034. 

 \bibitem{IOTW2}
  Kiyoshi Igusa, Kent Orr, Gordana Todorov, and Jerzy Weyman, \emph{Modulated semi-invariants}, arXiv: 1507.03051v2. 
  
\bibitem{PartII} Kiyoshi Igusa, \emph{Maximal green sequences for cluster-tilted algebras of finite type}, arXiv:1706.06503.  

\bibitem{InTh} Colin Ingalls and Hugh Thomas, \emph{Noncrossing partitions and representations of quivers}, Compos. Math. {\bf145} (2009), no. 6, 1533--1562. 

\bibitem{Kase} Ryoichi Kase, \emph{Remarks on lengths of maximal green sequences for quivers of type $\tilde A_{n,1}$}, arXiv:1507.02852. 

\bibitem{Keller} {Bernhard Keller}, \emph{On cluster theory and quantum dilogarithm identities}, Representations of algebras and related topics, EMS Ser. Congr. Rep., Eur. Math. Soc., Z\"urich, 2011, pp. 85--116.

\bibitem{KD} Laurent Demonet, Bernhard Keller, \emph{A survey on maximal green sequences}, arXiv:1904.09247. 
.
\bibitem{KQ}Alastair King and Yu Qiu, \emph{Exchange graphs and Ext quivers}, Adv. Math. {\bf285} (2015), 1106--1154. 

 \bibitem{King} 
 Alastair D. King, \emph{Moduli of representations of finite dimensional algebras}, Quart. J. Math. Oxford (2), \textbf{45} (1994), 515--530. 
 
 
 \bibitem{KS} Maxim Kontsevich and Yan Soibelman, \emph{Stability structures, motivic Donaldson-Thomas invariants and cluster transformations}, arXiv:0811.2435. 
 
 \bibitem{LL} Fang Li, Siyang Liu, \emph{On maximal green sequences in abelian length categories}, arXiv:1910.06510. 


\bibitem{NZ} 
Tomoki Nakanishi, Andrei Zelevinsky, \emph{On tropical dualities in cluster algebras}, Contemp. Math. \textbf{265},  217--226. 

\bibitem{Q}Yu Qiu, \emph{Stability conditions and quantum dilogarithm identities for Dynkin quivers}, Adv. Math. {\bf269} (2015), 220--264. 



\bibitem{Reading} Nathan Reading, \emph{Universal geometric cluster algebras}, Reading, Nathan. "Universal geometric cluster algebras." Mathematische Zeitschrift 277.1-2 (2014): 499-547.

\bibitem{R} Markus Reineke, \emph{The Harder-Narasimhan system in quantum groups and cohomology of quiver moduli}, Invent. Math. {\bf152} (2003), no. 2, 349--368. 


\bibitem{Rudakov} Alexei Rudakov, \emph{Stability for an abelian category}, J. of Algebra {\bf 197} (1997), 231--245. 

\bibitem{T} R.P. Thomas, \emph{Stability conditions and the braid group}, arXiv:0112214v5. 


\bibitem{Z} Alfonso Zamora, \emph{On the Harder-Narasimhan filtration for finite dimensional representations of quivers}, Geom. Dedicata, 170 (2014), 185--194. 

\end{thebibliography}
\end{document}